\documentclass[12pt]{amsart}

\usepackage{amsmath}
\usepackage{amssymb}
\usepackage{amsthm}
\usepackage[utf8]{inputenc}
\usepackage{booktabs}
\usepackage{xcolor}
\usepackage{graphicx}
\usepackage{hyperref}
\usepackage{subcaption}
\usepackage{float}
\usepackage{nicefrac}
\usepackage{tikz}
\usetikzlibrary{positioning,shapes.geometric, arrows}
\usepackage[percent]{overpic}
\usepackage{mathtools}

\theoremstyle{definition}
\newtheorem{Def}{Definition}[section]

\newtheorem{Exe}[Def]{Example}

\newtheorem*{Conv}{Conventions}

\newtheorem{Rem}[Def]{Remark}

\theoremstyle{theorem}
\newtheorem{The}[Def]{Theorem}
\newtheorem{Pro}[Def]{Proposition}
\newtheorem{Cor}[Def]{Corollary}
\newtheorem{Lem}[Def]{Lemma}

\newcommand{\R}{\mathbb{R}}
\newcommand{\C}{\mathbb{C}}
\newcommand{\K}{\mathbb{K}}
\newcommand{\N}{\mathbb{N}}
\newcommand{\Z}{\mathbb{Z}}

\newcommand{\Sph}{\mathbb{S}}
\newcommand{\lcs}{\frak{lcs}}

\newcommand{\nocontentsline}[3]{}
\newcommand{\tocless}[2]{\bgroup\let\addcontentsline=\nocontentsline#1{#2}\egroup}
\title[Derived sheaves in $\lcs$ geometry]{Derived sheaves in locally conformally symplectic geometry}

\author{Adrien Currier}
\address{Peking University, Beijing International Center for Mathematical Research,	Beijing, China}
\email{adrien.currier@pku.edu.cn}
\date{}

\begin{document}
	\maketitle
	\begin{abstract}
		In this paper, we use derived sheaves to study rigidity phenomena in the cotangent bundles of manifolds endowed with some locally conformally symplectic ($\lcs$) structure. Taking inspiration from the work  of Guillermou, Kashiwara and Schapira, we define a quantization for ``$\lcs$'' Hamiltonian isotopies, as well as new quantities: the asymptotic Betti numbers of a sheaf. We then show that those quantities are ``well-behaved'' with respect to said quantization and use this to give a sheaf-theoretical proof of the Chantraine-Murphy theorem. We also consider the quantization in light of the Tamarkin morphism and the displacement energy of sheaves. This allows us to derive a non-squeezing theorem for $\lcs$ geometry that is similar, although not identical, to the one recently proven by Bertelson, Chakravarthy, and Sandon. Indeed, the result shown in this paper is more in line with the contact non-squeezing theorem proven by Eliashberg, Kim and Polterovich in 2006.
	\end{abstract}
	\tableofcontents
	\section{Introduction}
	
	Locally conformally symplectic ($\lcs$) geometry is a generalization of symplectic geometry in which transition maps are taken to be symplectomorphisms of $\mathbb{R}^{2n}$ up to a positive constant. Although first introduced by H.-C. Lee in the 40s (see \cite{Lee1943AKO}), the subject remained in hibernation until being picked up by Vaisman in the 70s (see \cite{Vaisman1976OnLC}), to whom we owe the name. An advantage of allowing the transition maps to only be symplectic up to a factor is that $\lcs$ structures are ``flexible'', very much akin to contact structures in this regard. Indeed, strictly more manifolds can be endowed with $\lcs$ structures than can be endowed with symplectic structures (see \cite{eliashberg2020making} or \cite{Bertelson2021} for more details; see \cite{Angella2017StructureOL} for specific examples). However, said ``flexibility'' also means that rigidity results are more difficult to obtain. This issue is best illustrated in \cite{currier2025projectionexactlagrangianslocally} with the construction of exact Lagrangians (in the $\lcs$ sense) in cotangent bundles of closed manifolds whose Euler characteristic is different from that of the base manifold. This is to compare with a result by Abouzaid and Kragh (see \cite{Abouzaid2020TwistedGF}) that shows that, in the symplectic setting, the canonical projection restricted to an exact Lagrangian is a simple homotopy equivalence. At the core of the construction of those exact ``$\lcs$'' Lagrangians is the fact that Legendrians in $J^1M$ can be lifted to exact Lagrangians in $T^*(M\times\Sph^1)$. 
	
	More generally, adapting classical tools from symplectic geometry to the $\lcs$ setting is not a straightforward process. For example, as pointed out in \cite{Chantraine2016ConformalSG} by B. Chantraine and E. Murphy, naive adaptations of Floer theory to this setting will tend to run into some problems due to the failure of Gromov's compactness theorem for pseudo-holomorphic curves. It should be noted that some progress has been made in that direction in \cite{Oh2023PCLCM} for some specific cases. However, there have been some successes in adapting generating functions to prove rigidity results in $\lcs$ geometry (an approach pioneered in \cite{Chantraine2016ConformalSG}, and used in \cite{currier2025morsenovikovhomologybetacriticalpoints} and \cite{bertelson2026nonsqueezingglobalrigidityresults}). This makes sheaf theory an attractive option to tackle those kinds of problems. 
	
	Those three papers illustrate three different philosophies for the use of generating functions. In \cite{Chantraine2016ConformalSG}, the authors leverage Morse-Novikov homology, which is a homology theory tailor-made to study the $1$-form $d_\beta 1$, where $d_\beta$ is the Lichnerowicz derivative associated to $\beta$ (see definition \ref{Lichnerowicz}). In their paper, Chantraine and Murphy carefully exploit this homology and cut the base manifold to study the more general $d_\beta F$ for a generating function $F$. Their approach partially relies on the fact that exact ``$\lcs$'' Lagrangians and ``$\lcs$'' Hamiltonian isotopies in $T^*M$ can be lifted to Legendrians and contact isotopies in $J^1M$. They then use Chekanov's persistence theorem for generating functions in contact geometry. In \cite{bertelson2026nonsqueezingglobalrigidityresults},  Bertelson, Chakravarthy, and Sandon take inspiration from the study of contact diffeomorphisms to define the notion of generating functions for $\lcs$ Hamiltonian diffeomorphisms using the $\lcs$ box product, as defined in \cite{chantraine2024productslocallyconformalsymplectic}. However, they eschew Morse-Novikov homology in favor of singular homology, motivated by their use of spectral selectors for generating functions. Finally, in \cite{currier2025morsenovikovhomologybetacriticalpoints}, the author prefers to lift the $\lcs$ objects of $T^*M$ to symplectic objects in some cover of $T^*M$. The author then shows that there is a way of leveraging Morse theory (as opposed to Morse-Novikov theory) to study generating functions. As we will see in this paper, this approach can be adapted to the sheaf-theoretical setting. Indeed, this paper might be best understood as bridging the approaches of \cite{Chantraine2016ConformalSG} and \cite{bertelson2026nonsqueezingglobalrigidityresults} using the insights of \cite{currier2025morsenovikovhomologybetacriticalpoints} and of remark \ref{bertelson&co}, all in a sheaf-theoretic context. It is also the author's hope that a good sheaf theory might shed some light on how Floer theory should be adapted to this setting.
	
	While this last point is, for now, pure speculation, a sheaf theory for $\lcs$ geometry is of interest in its own right. Indeed, for the last 20 years or so, sheaf theory has been successfully used in symplectic geometry to prove results that had previously needed some Floer-theoretical touch (see, for example, \cite{Guillermou2012SQHIQNP} and \cite{Guillermou2019SheavesAS}), and has inspired some new strategies for the study of rigidity in symplectic geometry.\newline
	
	The aim of this paper is to showcase techniques to study $\lcs$ geometry via sheaf theory. To this end, we will provide a sheaf-theoretical proof of the following theorem, originally proven by Chantraine and Murphy using generating functions:
	
	\begin{The}\label{thm1}(\cite{Chantraine2016ConformalSG})
		Let $M$ be a closed manifold, $\beta\in\Omega^1(M)$ be closed, and $\lambda$ be the canonical Liouville form on $T^*M$. Then, for a generic Hamiltonian diffeomorphism $\phi$ of $(T^*M,\lambda,\beta)$, we have that
		\[\#\phi(M)\pitchfork M\geq\sum_i rank(HN_i(M,\beta;\R)),\]
		where $HN$ stands for Morse-Novikov homology.
	\end{The}
	
	More precisely, we will show that $\#\phi(M)\pitchfork M\geq\sum_j  b_j^\beta (\R_M)$, where $ b_j^\beta $ is the asymptotic Betti number associated to $\beta$ (see definition \ref{defNovikov}), and theorem \ref{thm1} is then derived from the fact that, as established in \cite{currier2025morsenovikovhomologybetacriticalpoints}, $ b_j^\beta (\R_M)\geq rank(HN_j(M,\beta))$. 
	
	We will also show a result similar, although not identical, to a result proven in \cite{bertelson2026nonsqueezingglobalrigidityresults} via generating functions:
	
	\begin{The}\label{thm3}
		Let $R_1$ and $R_2$ be such that, for some $k\in \N$, $\pi R^2_1\geq k\geq \pi R^2_2$. Define $D_{R_2}$ as the disc of radius $R_2$ in $\R^2$. Then there is no compactly supported Hamiltonian isotopy $\phi$ of $(\R^{2n}\times\Sph^1_z\times\Sph^1_\zeta,\lambda_{\R^n}+dz,d\zeta)$ such that \[\phi_1\left(\overline{\mathbb{B}_{R_1}^{2n}\times\Sph^1_z\times\Sph^1_\zeta}\right)\subset D_{R_2}\times\R^{2n-2}\times\Sph^1_z\times\Sph^1_\zeta.\]
	\end{The}
	
	This result is to be considered in light of \cite{Eliashberg2005ErratumT}, in which the authors prove a similar result in contact geometry. It should be noted that the techniques used in this present paper are not sheaf-theoretical adaptations of the techniques involving generating functions presented in any of the aforementioned papers. Instead, the proofs presented here will rely on the ideas found in \cite{Guillermou2012SQHIQNP}, \cite{Guillermou2019SheavesAS}, and \cite{currier2025morsenovikovhomologybetacriticalpoints}. Two objects are at the core of this paper: the asymptotic Betti number of a sheaf, and the equivariant homogeneous lift of a ($\lcs$) Hamiltonian isotopy.
	
	\begin{Conv} Throughout the rest of this paper, we will use the following conventions:
		\begin{enumerate}
			\item $X$ will be a connected manifold and, given any manifold $X$, $\beta\in\Omega^1(X)$ will be a closed form on $X$.
			\item $\tilde{X}$ will be the integral cover of $\beta$, $\pi:\tilde{X}\rightarrow X$ will be the canonical  projection, and $\mathcal{H}$ will be the group of automorphisms of $\tilde{X}$. Note that $\mathcal{H}$ is a free $\mathbb{Z}$-module, which we will assume to be of finite dimension.
			\item $(\alpha_1,\ldots,\alpha_r)$ is a basis of $\mathcal{H}$ such that, for all $i$, $\langle[\beta],\alpha_i\rangle>0$ when viewing $\alpha_i$ as an element of $H_1(X,\mathbb{Z})$.
			\item $\mathcal{H}_k=\left\{\sum_i \lambda_i\alpha_i\,:\,\lambda_i\in\mathbb{Z},|\lambda_i|\leq k\right\}.$ is the hypercube of side length $(2k+1)$ centered on $0$ in $\mathcal{H}$. Its boundary will be noted $\partial\mathcal{H}_k$.\end{enumerate}
	\end{Conv}  
	
	As suggested by those conventions, we will mostly be working on the integral cover of $\beta$. Indeed, asymptotic Betti numbers need to be considered with respect to some $\beta$.
	
	\begin{Def}[\textbf{asymptotic Betti number}]\label{defNovikov}
		Take $V$ the closure of a connected fundamental domain (with corners) such that the pullback of  $\beta$ on $\tilde{X}$ (also called $\beta$) is never in the cotangent of $\partial V$, in the sense that around any point of $\partial V-\partial\partial V$, there is a locally defined vector field $Y$ that is normal to $\partial V$ for some Riemannian metric, and such that $\beta(Y)\neq0$. Assume that $V$ is subanalytic. We can pick a basis $(\alpha_1,\ldots,\alpha_r)$ for $\mathcal{H}$ such that for every face $F$ in $\partial V$, either $\alpha_i(F)\subset \partial V$ or $\alpha^{-1}_i(F)\subset \partial V$. Consider $W_k=\bigcup_{\alpha\in\mathcal{H}_k}\alpha(V)$. Then, for $\K=\Z$ or $\K$ a field, 
		and for $F\in D^{b}(\K_{X})$, writing $b_j(F)=\text{rank}(
	H^j(R\Gamma(X;F))
		)$,
		\[ b_j^\beta (F):=\underset{k\rightarrow+\infty}{\lim}\;\frac{b_j\left((\pi^{-1}F)_{W_k}\right)}{(2k+1)^r},\]
		when the limit exists, is called the $j$-th asymptotic Betti number of $F$ with respect to $\beta$.
	\end{Def}
	\begin{Pro}\label{prop1}
		Let $F\in D^b(\K_{X})$ be a $\R$-constructible sheaf with compact support. Then the asymptotic Betti numbers of $F$ are well defined and do not depend on the arbitrary choices made during the selection of $V$.
	\end{Pro}

	It will be convenient to fix a good fundamental domain satisfying the properties given in the previous definition. The construction of such a fundamental domain can be found in \cite{Farber2004TopologyOC} and \cite{currier2025morsenovikovhomologybetacriticalpoints}.
	
	The proof of theorem \ref{thm1} will have two main ingredients. First will be proposition \ref{MorseFaisceaux}, which is a fairly straightforward adaptation of the classical microlocal Morse lemma and of the arguments found in \cite{currier2025morsenovikovhomologybetacriticalpoints}. The second ingredient will be the following technical theorem. In short, it states that Hamiltonian isotopies (of $\lcs$ type) in cotangent bundles are, in a sense, quantizable, and that their quantization behaves well with respect to the lift of sheaves to the integral cover of $\beta$. This quantization will be central to the proof of theorem \ref{thm3}.
	
	\begin{The}\label{thm2} For any manifold $X$, the cotangent bundle of $X$ minus the $0$-section will be written $\dot{T}^*X$.
		
		Let $M$ be a closed connected  manifold, $\beta\in\Omega^1(M)$ be closed, and $\lambda$ the canonical Liouville form on $T^*M$. We define $\beta'$ as the pullback of $\beta $ to $M\times\R$. Call $\tilde{M}$ the integral cover of $\beta$, and take $g$ a primitive of $\beta$ on $\tilde{M}$. Take $\pi:\tilde{M}\rightarrow M$ the canonical projection and $\pi_g(x,\xi):=D\pi(x,e^{g}\xi)$ where $D\pi:T^*\tilde{M}\rightarrow T^*M$. 
		
		Given $\phi$ an Hamiltonian isotopy of $(T^*M,\lambda,\beta)$ with compact support, and $\tilde{\Phi}$ its lift to $T^*\tilde{M}$ by $\pi_g$, then there is a map $ \rho_{eq}: T^*\tilde{M}\times\dot{T}^*\mathbb{R}\rightarrow  T^*\tilde{M}$ and $\Phi^{eq}$ a homogeneous Hamiltonian isotopy of $(\dot{T}^*(\tilde{M}\times\R),\lambda_{\tilde{M}\times\R},0)$, for $\lambda_{\tilde{M}\times\R}$ the canonical Liouville form, such that:
		\begin{enumerate}
			\item[i)] When restricting $\Phi^{eq}$ to $T^*\tilde{M}\times\dot{T}^*\mathbb{R}$ the following commutative diagram is commutative: \[\begin{array}{ccc}
				T^*\tilde{M}\times\dot{T}^*\mathbb{R}\times I & \overset{\Phi^{eq}_t}{\rightarrow}& T^*\tilde{M}\times\dot{T}^*\mathbb{R}\\
				\rho_{eq}\times id\downarrow& \circlearrowleft&\downarrow\rho_{eq}\\
				T^*\tilde{M}\times I&\overset{\tilde{\Phi}_t}{\rightarrow}&T^*\tilde{M}
			\end{array}\]
			\item[ii)] $\Phi^{eq}$ can be quantized (in the classical sense) by a sheaf $K^{eq}\in D^{lb}(A_{\tilde{M}\times\R\times\tilde{M}\times\R\times I})$ such that for any sheaf $F\in D^b(A_{M\times\R})$ with compact support and any compact $J\subset I$, there is a sheaf $G\in D^b(A_{M\times\R\times J})$ with compact support such that : \[K_J^{eq}\circ\big( (\pi\times id_\R)^{-1}F\big)\simeq(\pi\times id_\R\times id_J)^{-1}G; \]
			and for any $t\in J$: \[K^{eq}_t\circ\big( (\pi\times id_\R)^{-1}F\big)\simeq(\pi\times id_\R)^{-1}G_t; \]
			\item[iii)] Assuming $\Phi^{eq}$ is subanalytic. For any $s,t\in J$, any $j\in\N$, and any $\R$-constructible sheaf $F\in D^b(A_{M\times\R})$ with compact support, if $A$ is a field, and $G$ is as defined above, then \[ b_j^\beta (G_{s})= b_j^\beta (G_{t}).\]
		\end{enumerate}
	\end{The}
	
	The outline of this paper is as follows. Section \ref{towardsasheaf} will be mainly concerned with preliminary results. In this section, we will prove proposition \ref{prop1}, thus showing that asymptotic Betti numbers are well-defined for $\R$-constructible sheaves with compact support. The section will then conclude with the proof of a version of the microlocal Morse lemma adapted to asymptotic Betti numbers. Section \ref{onlcsgeometry} will focus on providing the reader with a short overview of the various objects in $\lcs$ geometry that will be studied in this paper. In section \ref{sectionLifts}, we will see various methods of transforming the various $\lcs$ objects into their symplectic counterparts. This section will contain a proof of theorem \ref{thm2}. Section \ref{sectionChantraineMurphy} will be concerned with a sheaf-theoretical proof of theorem \ref{thm1}. Finally, after a brief reminder of the standard sheaf-theoretical proof of the symplectic non-squeezing theorem in section \ref{non-squeezingsymp}, we will conclude with a short proof of theorem \ref{thm3} in section \ref{non-squeezing}.

	\subsection*{Acknowledgments}
	
	The author would like to thank Baptiste Chantraine for his feedback. The author especially thanks St\'ephane Guillermou for helpful comments and discussions.
	
		\section{Equivariant sheaves and asymptotic Betti numbers}\label{towardsasheaf}
		
 	We use the notations of \ref{prop1}. For $\bullet=+,-,lb,b$, call $D^\bullet_{c,\, \R-const}(\K_X)$ the subcategory of $D^\bullet(\K_X)$ of $\R$-constructible sheaves $F$ with compact support. Then $\pi^{-1}(D^\bullet_{c,\, \R-const}(\K_X))$ is a subcategory of $D^\bullet(\K_{\tilde{X}})$ whose objects are sheaves $F$ such that:
 	\begin{enumerate}
 		\item[1)] for any deck transformation $\alpha$, we have $\alpha^*F\simeq F$;
 		\item[2)] $\pi(supp(F))$ is compact;
 		\item[3)] $F$ is $\R$-constructible.
 	\end{enumerate}

	\begin{Conv}
	 As a short-hand, any sheaf satisfying the first point will be said to be equivariant under deck transformations. Any sheaf satisfying the three points will be called a $\beta$-sheaf.
	\end{Conv}
	
	\subsection{Proof of proposition \ref{prop1}}
	Keeping in mind the definitions given in the introduction, let us begin our exposition of a sheaf theory for $\lcs$ geometry by proving proposition \ref{prop1}. 
	
	\begin{Lem}\label{lem1}
		Let $F$ be $\beta$-sheaf, and $(W_k)_k$ a compact exhaustion as in definition \ref{defNovikov}. Then, for some $b\in \R$,	\[b_j(F_{W_{k}})=b(2k+1)^r+O_{+\infty}\left((2k+1)^{r-1}\right).\]
	\end{Lem}
	
	\begin{proof}
		Recall that $W_k=\bigcup_{\alpha\in\mathcal{H}_k}\alpha(W_0)$. Consider $U_0$ an open neighborhood of $W_0$ and $U_k=\bigcup_{\alpha\in\mathcal{H}_k} \alpha(U_0)$ a neighborhood of $W_k$. Since $\alpha^*F\simeq F$, for any $k>0$, the quantity \[s_{n,j}=\sup\left\{rk\left(H_j\left(\bigcap_{\alpha\in I}\alpha(W_0),F\right)\right)\,|\, I\subset\mathcal{H}_k,\,\# I=n\right\},\]
		does not depend on $k$ for $k$ big enough. Then successive applications of Mayer-Vietoris imply that \[b_j(F_{W_{k}})\leq \#\mathcal{H}_k\times\sum_{n\geq1}c_ns_{n,j},\]
		where $c_n$ is the number of faces of codimension $n-1$ in $W_0$ (and its boundary).
		
		More generally, $W_{2km+(k+m)}$ can be viewed as $(2m+1)^r$ copies of $W_k$ glued along their faces. The same Mayer-Vietoris sequence implies that \[|b_j(F_{W_{2km+(k+m)}})-(2m+1)^rb_j(F_{W_{k}})|\leq(2m+1)^r\times\sum_{n\geq2}c_n(s'_{n,j}+s'_{n,j-1}),\]
		where $s'_{n,j}$ is defined similarly to $s_{n,j}$ using $W_k$ instead of $W_0$.
		Indeed, the distinguished triangle $H^j(U_1\cup U_2,F)\rightarrow H^j(U_1\sqcup U_2,F)\rightarrow H^j(U_1\cap U_2,F)$, for two open sets $U_1$ and $U_2$, implies that
		\begin{align*}
		|rk\left(H^j(U_1\cup U_2,F)\right)&-rk\left(H^j(U_1\sqcup U_2,F)\right)|\\
		&\leq rk\left(H^j(U_1\cap U_2,F)\right)+rk\left(H^{j-1}(U_1\cap U_2,F)\right)
		\end{align*}
		
		Moreover, $s'_{2,j}\leq (\#\partial \mathcal{H}_k) \sum_{n\geq2}c_ns_{n,j}$, $s'_{3,j}\leq (\#\partial\partial \mathcal{H}_k) \sum_{n\geq3}c_ns_{n,j}$, etc. Therefore, there exists some constant $cst$ such that
		\[|b_j(F_{W_{2km+(k+m)}})-(2m+1)^rb_j(F_{W_{k}})|\leq(2m+1)^r(2k+1)^{r-1}cst,\]
		which implies
		\begin{align}\label{prf1}
		|\frac{b_j(F_{W_{2km+(k+m)}})}{(2m+1)^r(2k+1)^r}-\frac{b_j(F_{W_{k}})}{(2k+1)^r}|\leq\frac{cst}{2k+1}
		\end{align}
		More generally, for any $p>0$, there is an $m$ such that $2km+(k+m)\leq k+p\leq 2k(m+1)+(k+m+1)$. The distinguished triangle $F_{W_{2k(m+1)+(k+m+1)}/W_{k+p}}\rightarrow F_{W_{2k(m+1)+(k+m+1)}}\rightarrow F_{W_{k+p}}$ and the previous Mayer-Vietoris arguments imply that
		\begin{align}\label{prf2}
		|b_j(F_{W_{2k(m+1)+(k+m+1)}})&-b_j(F_{W_{k+p}})|\nonumber \\
		&=O_{+\infty}\left((2m+3)^r(2k+1)^r-(2(k+p)+1)^r\right)\nonumber \\
		&=O_{+\infty}\left(((2m+3)^r-(2m+1)^r)(2k+1)^r\right)
		\end{align}
		Dividing by $(2m+1)^r(2k+1)^r$ in (\ref{prf2}), and in light of (\ref{prf1}), we have $\limsup_k\left(\frac{b_j(F_{W_{k}})}{(2k+1)^r}\right)=\liminf_k\left(\frac{b_j(F_{W_{k}})}{(2k+1)^r}\right)=:b$ and $|b-b_j(F_{W_{k}})/(2k+1)^r|\leq cst/(2k+1)$.
	\end{proof}
	
	\begin{Lem}\label{lem2}
		Let $F$ be $\beta$-sheaf, and $(W^1_k)_k$ and $(W^2_k)_k$ be two compact exhaustions as in definition \ref{defNovikov}. Then \[b_j(F_{W^1_{k}})=b_j(F_{W^2_{k}})+O_{+\infty}\left((2k+1)^{r-1}\right).\]
	\end{Lem}
	
	\begin{proof} Given the assumptions, we may as well assume that $X$ is compact. To see why, one may simply take the restriction to a neighborhood of the projection of the support of $F$, together with neighborhoods of the loops that form a (finite) basis for the coimage of 
		\begin{align*}
			<[\beta],\cdot>:&\pi_1(X)\rightarrow \R\\
			\gamma&\mapsto\int_\gamma\beta
		\end{align*}
		
		Let $W_0^1$ and $W_0^2$ be two possible choices for the fundamental domain $V$. Since $F$ is equivariant under deck transformations, we can assume that $int(W_0^1)\cap int(W_0^2)\neq\varnothing$ without loss of generality. Define $W^1_k=\bigcup_{\alpha\in\mathcal{H}_k}\alpha(W_0^1)$ and $W^2_k=\bigcup_{\alpha\in\mathcal{H}_k}\alpha(W_0^2)$. Since both $W_0^1$ and $W_0^2$ are subanalytic, so is $W^1_k$ and $W^2_k$. This implies, in turn, that 
		$F_{W^1_k}$ and $F_{W^2_k}$ are also $\R$-constructible.\newline
		
		First, let us assume that $W^1_k\subset int (W^2_{k+1})$ and $W^2_k\subset int(W^1_{k+1})$. To tidy up the notations, take $F_k^1:=F_{W^1_k}$ and $F_k^2:=F_{W^2_k}$. Take $V^k$ some arbitrarily small open neighborhood of $\overline{W^1_{k+1}\backslash W^2_{k}}$. Since for any neighborhood $U$ of $W^1_{k+1}$, we have $R\Gamma(U,F_{W^1_k})\simeq R\Gamma(W^1_{k+1},F_{W^1_k})$,  we have the Mayer-Vietoris long exact sequence:
		\begin{align*}
			\ldots\rightarrow R\Gamma(W^1_{k+1},F_{W^1_{k+1}})&\rightarrow R\Gamma(V^k,F_{W^1_{k+1}})\oplus R\Gamma(int(W^2_{k}),F_{W^1_{k+1}})\\
			&\rightarrow R\Gamma(V^k\cap int (W^2_{k}),F_{W^1_{k+1}})\overset{+1}{\rightarrow}\ldots
		\end{align*}
		
		And, similarly, we have
		\begin{align*}
			\ldots\rightarrow R\Gamma(W^1_{k+1},F_{W^2_{k}})&\rightarrow R\Gamma(V^k,F_{W^2_{k}})\oplus R\Gamma(int(W^2_{k}),F_{W^2_{k}})\\
			&\rightarrow R\Gamma(V^k\cap int(W^2_{k}),F_{W^2_{k}})\overset{+1}{\rightarrow}\ldots
		\end{align*}
		However, we have the obvious isomorphisms \[R\Gamma(int(W^2_{k}),F_{W^1_{k+1}})\simeq R\Gamma(int(W^2_{k}),F_{W^2_{k}})\] and \[R\Gamma(V^k\cap int(W^2_{k}),F_{W^1_{k+1}})\simeq R\Gamma(V^k\cap int(W^2_{k}),F_{W^2_{k}}).\] 
		
		Therefore, we only have to show two things:
		\begin{enumerate}
		\item $b_j \left(F_{W^1_{k+1}}\right)=b_j \left(F_{W^1_{k}}\right)+O_{+\infty}\left((2k+1)^{r-1}\right)$
		\item $rank \left(H^j(V^k,F_{W^1_{k}})\right)=rank \left(H^j(V^k,F_{W^2_{k}})\right)+O_{+\infty}\left((2k+1)^{r-1}\right)$
		\end{enumerate}
		
		 More precisely, we are going to show the second point by proving that those cohomology groups are both of rank at most $O_{+\infty}((2k+1)^{r-1})$.
		
		For $H^j(V^k,F_{W^1_{k+1}})$, first take a finite open cover $(O^k_i)_{i\in I_k}$ of $V^k$ where all $O^k_i$ are subanalytic and are deck transformations of one another. More precisely, we can find an open relatively compact subanalytic set $O_0$ such that $\cup_{\alpha\in\mathcal{H}_{k+1}\backslash\mathcal{H}_{k}}\alpha(O_0)\supset V^k$. Considering the subsets $\cup_{\alpha\in J}\alpha(O_0)\subset\cup_{\alpha\in\mathcal{H}_{k+1}\backslash\mathcal{H}_{k}}\alpha(O_0)$ for any $J\subset\mathcal{H}_{k+1}\backslash\mathcal{H}_{k}$, one can build a sequence of Mayer-Vietoris sequences yielding that \[rank \left(H^j(\cup_{\alpha\in\mathcal{H}_{k+1}\backslash\mathcal{H}_{k}}\alpha(O_0),F_{W^1_{k+1}}))\right)\leq cst\times(2(k+1)+1)^{r-1}).\] Indeed, $F_{W^1_{k+1}}$ is $\R$-constructible and $F$ equivariant under deck transformations, implying that $H^j(\bigcap_{\alpha\in J}\alpha(O_0),F_{W^1_{k+1}})\leq cst$ for any $J\subset\mathcal{H}_{k+1}\backslash\mathcal{H}_{k}$ and some constant $cst$ not depending on $J$. Moreover, the number of elements in $\mathcal{H}_{k+1}\backslash\mathcal{H}_{k}$ is proportional to $(2(k+1)+1)^{r-1})$.
		
		Also, note that the constant $cst$ can be defined independently of $k$ since, if for some $k_1,k_2\in\N$ and some $\alpha,\alpha'\in\mathcal{H}$, the subsets $\alpha(O_0)\cap W^1_{k_1}$ and $\alpha'(O_0)\cap W^1_{k_2}$ are deck transformations of one another, then $R\Gamma(\alpha(O_0),F_{k_1}^1)\simeq R\Gamma(\alpha'(O_0),F_{k_2}^1)$ (by invariance under deck transformations). Moreover, we have similar results for finite intersections of sets in $\big(\alpha(O_0)\big)_{\alpha\in\mathcal{H}_{k+1}\backslash\mathcal{H}_{k}}$. Therefore, since we may assume that $\cup_{\alpha\in\mathcal{H}_{k+1}\backslash\mathcal{H}_{k}}\alpha(O_0)= V^k$, \[rank(H^j(V^k,F_{W^1_{k+1}}))\leq cst\times(2(k+1)+1)^{r-1}).\]
		This strategy also yields that $rank\left(H^j(V^k,F_{W^2_{k}})\right)\leq cst\times (2(k+1)+1)^{r-1}$.

	Therefore, we have the following equality:
	\begin{enumerate}
		\item $b_j(F_{W^1_{k+1}})=b_j(F_{W^2_{k}})+O_{+\infty}((2k+1)^{r-1});$
		\item $b_j(F_{W^1_{k+1}})=b_j(F_{W^1_{k}})+O_{+\infty}((2k+1)^{r-1})$ by taking $W^1_0=W^2_0$.
	\end{enumerate} 
	This yields $b_j(F_{W^1_{k}})=b_j(F_{W^2_{k}})+O_{+\infty}((2k+1)^{r-1})$.
	
	Dropping the assumption $W^1_k\subset int (W^2_{k+1})$ and $W^2_k\subset int(W^1_{k+1})$, the result can be derived from the first case by taking a finite sequence of fundamental domains $V_i$, and the associated $(W^i_k)_k$ such that, for every $i$, $(W^i_k)_k$ and $(W^{i+1}_k)_k$ satisfy the hypothesis laid out in the previous case.
	\end{proof}
	
	\begin{proof}[proof of proposition \ref{prop1}]
	The two previous lemmas, together, prove the proposition.
	\end{proof}
	
	\begin{Rem}
	The strategy used in this proof can be extended to prove that, for example, $b_j(F_{\partial_- W_k})=O_{+\infty}\left((2k+1)^{r-1}\right)$ for any $j$, where $\partial_- W_k$ is the adherence of the subset of $\partial W_k$ over which $\beta$ points towards the interior of $W_k$. This implies that, $b_j(F_{W_k\backslash\partial_- W_k})=b_j(F_{W_k})+O_{+\infty}\left((2k+1)^{r-1}\right)$.
	\end{Rem}
	
	In the rest of this paper, we will use the equivalent definition for $\beta$-sheaves: \[ b_j^\beta (F)=\underset{k\rightarrow+\infty}{\lim}\frac{b_j(F_{W_k\backslash \partial_- W_k})}{(2k+1)^r}.\]
	
	Considering $F_{W_k\backslash \partial_- W_k}$ instead of $F_{W_k}$ has one main advantage: the micro-support behaves ``better'' under deck transformations, and $F_{W_k\backslash \partial_- W_k}$ is still $\R$-constructible since ${W_k\backslash \partial_- W_k}$ is subanalytic.
	
	Indeed, for any locally closed set $Z$ and closed subset $Z'$, the exact sequence $\ldots\rightarrow F_{Z\backslash Z'}\rightarrow F_Z\rightarrow F_{Z'}\overset{+1}{\rightarrow}\ldots$ implies that if $(x,\xi)\in T^*\tilde{X}$ such that $x\in\partial W_0$ and $(x,\xi)\in SS(F_{W_0\backslash \partial_- W_0})-SS(F)$, then there is a deck transformation $\alpha\neq id$ such that $\alpha(x)\in \partial W_0$ and $\alpha(x,\xi)\in SS(F_{W_0\backslash \partial_- W_0})$. It should be noted that if $x\in\partial W_0-\partial\partial W_0$, then there is a unique deck transformation $\alpha\neq id$ such that $\alpha(x)\in\partial W_0$.

	\begin{Conv}
	In the rest of this paper, $A$ will be a commutative field $\K$, except when specified.  Moreover, $M$ will always be a connected closed manifold, $\beta$ will always be a closed $1$-form on $M$, and its various pullbacks will also be called $\beta$ when confusion is unlikely. 
	
	Moreover, for any manifold $X$, $D^{\bullet}(\K_{X})$ will be denoted by $D^{\bullet}(X)$, where the dot stands for either $lb$ or $b$.
	
	By $M$, we will always denote a closed connected manifold.
	\end{Conv}

	\subsection{Microlocal Morse lemma for asymptotic Betti numbers}\label{amorsetheory}
	In the rest of the paper, it will be convenient to fix such a $V$ and not modify it. We will now give a construction that will yield one, using the beginning of Farber's construction for his universal chain complex (see \cite{Farber2004TopologyOC}). More specifically, we are interested in his construction of a ``good'' fundamental domain in a covering space of $X$.
	
	\begin{Rem}(\cite{Farber2004TopologyOC})
		Let $\beta$ be a closed $1$-form on $X$, then there is an $r\in\mathbb{N}$ and a map \[(\psi_1,\ldots,\psi_r):M\rightarrow\mathbb{S}^1_1\times\ldots\times\mathbb{S}^1_r\] such that $[\beta]=\sum_ia_i\psi_i^{*}[d\theta_i]$, where $d\theta_i$ is the canonical generator of $H^1(\mathbb{S}^1_i)$, the $\psi_i^{*}[d\theta_i]$ are linearly independent and the $a_i$ are linearly independent over $\mathbb{Z}$. 
	\end{Rem}
	Moreover, due to Thom's transversality, we have:
	\begin{Rem}(\cite{Farber2004TopologyOC})
		For some generic $\beta$, generic elements $c_1,\ldots,c_r\in\mathbb{S}^1$, and for every $i$, the various maps $\phi_i^{-1}(\{c_i\})$ are submanifold which intersect transversely with one another.
	\end{Rem}
	These considerations allow for the statement of the following definition:
	\begin{Def}(\cite{Farber2004TopologyOC})
		Let $U=X-\bigcup_i\phi_i^{-1}(\{c_i\})$, take $V$ in $\tilde{X}$ diffeomorphic to $U$ via the canonical projection. One may assume that $V$ is connected: take a tree in $X$ such that each connected component of $U$ contains one and only one vertex, then one can lift this tree to a tree in $\tilde{X}$ and take $V$ a preimage of $U$ containing the lift of the tree. Define $W_k=\cup_{\alpha\in\mathcal{H}_k}\alpha(\overline{V})$.
	\end{Def}
	
	\begin{Rem}
		If $X=T^*M$ (resp. $X=M\times\R^m$) and $\beta\in\Omega^1(M)$, then good fundamental domains used in this paper will be built as follows: first, fix a good fundamental domain $V\in \tilde{M}$, then $T^*V$ (resp. $V\times\mathbb{R}^m$) is a good fundamental domain for $\tilde{X}$.
	\end{Rem}
	
	We can now adapt proposition 5.4.20 in \cite{Kashiwara1990SheavesOM} to asymptotic Betti numbers:

	\begin{Pro}\label{MorseFaisceaux} Pick an $m\in \N$. Let $\beta\in\Omega^1(M)$ be closed and generic, and pull it back to $M\times\R^m$ through the canonical projection.
		
		Let $F\in D^b(\tilde{M}\times\R^m)$ be a $\beta$-sheaf. Write $F_k=F_{W_k\backslash\partial_-W_k}$. Let $g$ be a primitive of $\beta$ in $\tilde{M}\times\R^m$. Take a smooth map $f:M\times\R^m\rightarrow\R$ and define 
		\begin{align*}
			\psi:\tilde{M}\times\R^m&\rightarrow\R\\
			x&\mapsto e^{-g(x)}f\circ \pi(x)
		\end{align*}
		
		$\Lambda_\psi$ the $0$-exact Lagrangian given by:
		$\Lambda_\psi=\{(x,d\psi_x)\;:\; x\in \tilde{M}\times\R^m\}.$
		Index the projection on $\tilde{M}\times\R^m$ of the points $(x,\xi)\in\Lambda_\psi\cap SS(F_k)$ with $I_k$. Write $D^k_i=\{y:\psi(y)\geq\psi(x_i)\}$ for every $i\in I_k$.
		
		Assume that:
		\begin{itemize}
			\item $\{x_i\,:\,i\in I_k\}$ is finite.
			\item for all $t\in \R$ and $k\in\N$, $\{x\in \text{supp}(F_k): \psi(x\leq t)\}$ is compact.
			\item $A^k_i:=R\Gamma_{D^k_i}(F)_{|x_i}\in D^b(Mod^f(\K))$ for all $i\in I_k$ such that $x_i\in int(W_k)$.
			\item $B^k_i:=R\Gamma_{D^k_i}(F_k)_{|x_i}\in D^b(Mod^f(\K))$ for all $i\in I_k$ such that $x_i\in \partial W_k$.
			\item $\Lambda_\psi$ and $SS(F)$ do not intersect in the fibers of $T^*(\tilde{M}\times\R^m)$ over $\partial W_0=\partial\overline{V}$.
		\end{itemize}
		Then:
		\begin{itemize}
			\item $R\Gamma(\tilde{M}\times\R^m,F_k)\in D^b(Mod^f(\K))$ for any $k$.
			\item $(-1)^l\sum_{j\leq l}(-1)^j b_j^\beta (F)\leq (-1)^l\sum_{j\leq l}(-1)^j\left(\sum_ib_j(A^0_i)\right)$.
		\end{itemize}
	\end{Pro}
	\begin{Rem}
		Index the projection on $\tilde{M}$ of the points of $(x,\xi)\in\Lambda_\psi\cap SS(F)\cap T^* int(W_0)$ with the set $J$. Then for any $i\in I_0$ there is a (unique) $j\in J$ such that \[A^0_i\simeq R\Gamma_{\{\psi(y)\geq\psi(x_j)\}}(F)_{|x_j}.\] 
		
		In this equality, $F$ may be replaced with $F_{W_0}$.
	\end{Rem}
	\begin{proof}
		Since $SS(F_k)\subset T^*W_k$, we can use proposition 5.4.20 in \cite{Kashiwara1990SheavesOM} to deduce that $ R\Gamma(\tilde{M}\times\R^m,F_k)\simeq R\Gamma(W_k,F_k)\in D^b(Mod^f(\K))$ for any $k$, and that:
		\[(-1)^l\sum_{j\leq l}(-1)^jb_j(F_k)\leq (-1)^l\sum_{j\leq l}(-1)^j\left(\left[\sum_ib_j(A^k_i)\right]+\left[\sum_i b_j(B^k_i)\right]\right).\]
		Since $F$ is a $\beta$-sheaf, $b_j(\bigoplus_{i}A^k_i)=(2k+1)^rb_j(\bigoplus_{i}A^0_i)$. Indeed, for each $A^0_i$, there are exactly $(2k+1)^r$ derived chains of modules $A^k_i$ to which $A^0_i$ is isomorphic. The coefficient is equal to the volume spanned by the hypercube $\mathcal{H}_k$.
		
		Similarly, for each $B^0_i$ there is at most $cst\times (2k+1)^{r-1}$ isomorphic derived chains of modules  $B^k_i$. The constant depends only on the number of faces of codimension at least $1$ in the boundary of $\mathcal{H}_k$.
		
		Therefore, dividing by $(2k+1)^r$ and taking $k\rightarrow+\infty$ yields the result.
	\end{proof}
	
	\begin{Rem}
	The condition ``$F$ is a $\beta$-sheaf'' may be relaxed to  $F\in D^b(\tilde{M}\times\R^m)$ is invariant under deck tranformations, and $b_j^\beta(F)$ exists for every $j$. Moreover, if we don't know whether $b_j^\beta(F)$ exists or not, the proof still yields \[\limsup_{k\rightarrow+\infty}\frac{b_j(F_k)}{\#\mathcal{H}_k}\leq \sum_i b_j(A_i^0).\]
	\end{Rem}
	
	\section{On $\lcs$ geometry}\label{onlcsgeometry}
	
	\subsection{General definitions}
	Let us begin this section by stating that there are three equivalent definitions of $\lcs$ manifold. For additional information about those characterizations, the reader may refer themselves to \cite{Chantraine2016ConformalSG} and \cite{chantraine2024productslocallyconformalsymplectic}, or to \cite{currier2025projectionexactlagrangianslocally}. Here, we will pick the easier definition to work with in our setting.
	
	At the core of most definitions in $\lcs$ geometry is the derivative $d\cdot-\beta\wedge\cdot$, called Lichnerowicz derivative.
	\begin{Def}\label{Lichnerowicz}
		Let $M$ be a manifold and $\beta\in\Omega^1(M)$ be closed. Then the Lichnerowicz derivative associated to $\beta$ is the map:
		\begin{align*}
			d_\beta:\Omega^*(M)&\rightarrow\Omega^{*+1}(M)\\
			\alpha&\mapsto d\alpha-\beta\wedge\alpha
		\end{align*}
	\end{Def}
	One can easily verify that $d_\beta^2=0$, meaning that $(\Omega^*(M),d_\beta)$ is a proper chain complex. In the rest of this section, it will become apparent that the various definitions central to $\lcs$ geometry follow similar patterns: take a ``classical'' definition of symplectic geometry, but change the derivative $d$ to $d_\beta$.
	\begin{Def}
		Let $M$ be a manifold. Take $\omega\in\Omega^2(M)$ non-degenerate and $\beta\in\Omega^1(M)$ closed such that $d_\beta\omega=0$. The form $\omega$ will be called a $\lcs$ form, whereas $\beta$ will be called the Lee form. The pair $(\omega,\beta)$ will be called a $\lcs$ pair. By $LCS(M)$ we will denote the set of $\lcs$ pairs on $M$ quotiented by the equivalence relation $(\omega,\beta)\sim(e^{g}\omega,\beta+dg)$.
		
		A $\lcs$ manifold is the data of a manifold $M$ and an element of $LCS(M)$. For the sake of simplicity, we will often forgo writing the whole equivalence class and simply write $(M,\omega,\beta)$ where $(\omega,\beta)$ is an $\lcs$ pair in the chosen class.
	\end{Def}
	Observe that, for any differential form $\alpha$, $e^gd_\beta\alpha=d_{\beta+dg}(e^g\alpha)$. Since $\beta$ is closed and therefore locally exact, this implies that $\omega$ is locally symplectic up to a conformal factor. Therefore, as in symplectic geometry, there are some maximal (dimension-wise) submanifolds on which $\omega$ pulls back to $0$.
	\begin{Def}
		Let $(M,\omega,\beta)$ be a $\lcs$ manifold of dimension $2n$, and $L$ be a manifold of dimension $n$. An embedding (resp. immersion) $i:L\rightarrow M$ such that $i^*\omega=0$ is called a Lagrangian embedding (resp. immersion) and $i(L)$ is called a (resp. immersed) Lagrangian submanifold. The ``submanifold'' is sometimes dropped.
	\end{Def}
	Observe that since $d_\beta^2=0$, $\omega$ being $d_\beta$-closed means that it can be $d_\beta$-exact in some cases.
	\begin{Def}
		Let $(M,\omega,\beta)$ be a $\lcs$ manifold. If $\omega=d_\beta\lambda$ for some $\lambda\in\Omega^1(M)$, then the pair $(\lambda,\beta)$ will be called an exact $\lcs$ pair and $\lambda$ will be called an exact $\lcs$ form. By $ELCS(M)$ we will denote the set of exact $\lcs$ pairs on $M$ quotiented by the equivalence relation $(\lambda,\beta)\sim(e^{g}(\lambda+d_\beta f),\beta+dg)$ for some maps $f$ and $g$.
		
		An exact $\lcs$ manifold is the data of a manifold $M$ and an element of $ELCS(M)$. For the sake of simplicity, we will often forgo writing the whole equivalence class and simply write $(M,\lambda,\beta)$ where $(\lambda,\beta)$ is an exact $\lcs$ pair in the chosen class.
	\end{Def}

	Note that, given a $1$-form $\lambda$ and a closed $1$-form $\beta$, if $d_\beta\lambda$ is non-degenerate, then $(\lambda,\beta)$ is an exact $\lcs$ pair. Indeed, $d_\beta^2=0$.
	
	\begin{Exe}
		Given a manifold $X$ endowed with a contact form $\alpha$, the manifold $(X\times\Sph_\theta^1,\alpha,d\theta)$ is an exact $\lcs$ manifold. For a more specific example, $\Sph^3\times\Sph^1$ can be endowed with an exact $\lcs$ structure.
	\end{Exe}
	
	As in symplectic geometry, an $\lcs$ form $\omega$ can be used to define isotropic-coisotropic submanifolds called Lagrangian submanifolds. In exact $\lcs$ geometry, we can define a special class of Lagrangian submanifolds called exact Lagrangian submanifolds, which are direct generalizations of their symplectic counterparts.
	
	\begin{Def}
		Let $(M,\lambda,\beta)$ be an exact $\lcs$ manifold of dimension $2n$, and $L$ be a manifold of dimension $n$. An embedding (resp. immersion) $i:L\rightarrow M$ such that $i^*\lambda=d_{i^*\beta}f$ for some $f\in C^\infty(L)$ is called a $\beta$-exact Lagrangian embedding (resp. immersion) and $i(L)$ is called an  (resp. immersed) exact Lagrangian submanifold. The ``submanifold'' is sometimes omitted. Whenever $\beta$ does not matter or is implicit, ``$\beta$-exact'' will just be written as ``exact''.
	\end{Def}
	
	In this paper, we will endeavor to study the rigidity of exact Lagrangians. As such, we need a good notion of deformations for such submanifolds.
	
	\begin{Def}\label{defham}
		Let $(M,\lambda,\beta)$ be an exact $\lcs$ manifold and\newline set $d_\beta\lambda=:\omega$.
		\begin{itemize}
			\item For any map smooth $h:M\times[0,1]\rightarrow\mathbb{R}$, the time-dependent vector field $X^t_h$ uniquely determined by the equation \[\iota_{X^t_h}\omega_x=-d_\beta h_{(x,t)} \] is called the Hamiltonian vector field associated to $h_t$.
			\item The vector field $X^t_h$ generates a flow $\phi^t_h$ such that 
			\[\frac{d}{dt}\phi^t_h=X^t_h\text{ and } \phi_0=id.\]
			Whenever $\phi_t$ is a diffeomorphism at every point in time, $\phi_t$ is called the Hamiltonian flow associated with $h$.
			\item For $\phi_t$ a Hamiltonian flow, for any $t_0\in[0,1]$, the diffeomorphism $\phi_{t_0}$ is called a Hamiltonian diffeomorphism.
		\end{itemize}
	\end{Def}
	Careful computation yields that Hamiltonian flows of an exact $\lcs$ manifold $(M,\lambda,\beta)$ preserve the class of the exact $\lcs$ pair $(\lambda,\beta)$.
	
	\subsection{The $\lcs$ geometry of cotangent bundles}
	
	In the rest of this paper, we will be interested in $\lcs$ manifolds of the form $(T^*M, \lambda_M,\beta)$ where $\lambda_M$ is a canonical Liouville form and $\beta\in\Omega^1(M)$ is closed (we omitted writing the canonical pullback of $\beta$ to $T^*M$ for clarity's sake). One should note that for such a choice of an exact $\lcs$ pair, $d_\beta\lambda_M$ is always non-degenerate. 
	
	Moreover, any choice of exact $\lcs$ pair $(\lambda_M,\gamma)$ for $\gamma\in\Omega^1(T^*M)$ is equivalent to an exact $\lcs$ pair $(e^{-h}\lambda_M,\beta)$ for some smooth map $h$ on $T^*M$ and some closed $\beta\in\Omega^1(M)$. As shown in \cite{currier2025projectionexactlagrangianslocally}, if $h(x,t\frac{\xi}{\|\xi\|})=o_{+\infty}(\ln(t))$ for some metric and $t\in\R_+^*$, then there is a diffeomorphism $\phi:T^*M\rightarrow T^*M$ such that $\phi^*\lambda_M=e^{-h}\lambda_M$ and $\phi^*\beta=\beta$.
	
	\section{Lifting $\lcs$ objects to symplectic objects}\label{sectionLifts}
	
	From the previous subsection, one should note that if $(\lambda,dg)$ is an exact $\lcs$ pair for some smooth map $g$, then so is $(e^{-g}\lambda,0)$, i.e., we are in the classical symplectic setting. More specifically, if $\lambda$ is the canonical Liouville form on $T^*M$ and $g$ is a smooth map on $M$, the re-scaling $(T^*M,\lambda,dg)\rightarrow (T^*M,\lambda,0);\;(x,\xi)\mapsto(x,e^{-g}\xi)$ allows us to view our $\lcs$ manifold as a symplectic manifold. Applying this insight to a more general closed $1$-form, call it $\beta$, on $M$ is as simple as pulling everything back to the integral cover of $\beta$ through the projection:
	\begin{align*}
		\pi_g:\;(T^*\tilde{M},\lambda_{\tilde{M}},0)&\rightarrow(T^*M,\lambda_M,\beta);\\
		(x,\xi)&\mapsto (\pi(x), D\pi_x(e^{g(x)}\xi))
	\end{align*}
	
	where $\tilde{M}$ is the integral cover of $\beta$, $g$ is the primitive of $\beta$ on $\tilde{M}$, and $\pi:\tilde{M}\rightarrow M$ is the canonical projection.
	
	This implies that if $\Lambda$ is an exact Lagrangian submanifold and $\phi_t$ is a Hamiltonian isotopy of $(T^*M,\lambda_M,\beta)$, then:
	\begin{enumerate}
		\item $\pi_g^{-1}(\Lambda)$ is an exact Lagrangian submanifold of $(T^*\tilde{M},\lambda_{\tilde{M}},0)$;
		\item let $\tilde{\Phi}$ be Hamiltonian isotopy of $(T^*\tilde{M},\lambda_{\tilde{M}},0)$ given by the Hamiltonian $\tilde{f}_t=e^{-g}\times(f_t\circ\pi_g)$, then there is a commutative diagram:
		\[\begin{array}{ccc}
			T^*\tilde{M}\times I & \overset{\tilde{\Phi}_t}{\rightarrow}& T^*\tilde{M}\\
			\pi_g\times id\downarrow& \circlearrowleft&\downarrow\pi_g\\
			T^*M\times I&\overset{\phi_t}{\rightarrow}&T^*M
		\end{array}\]
	\end{enumerate}
	
	The first point is a direct consequence of the previous discussion, while the second one can be inferred from computing $\lambda_{\tilde{M}}=e^{-g}\pi_g^*(\lambda_M)$ and therefore $d\lambda_{\tilde{M}}=e^{-g}\pi_g^*(d\lambda_M-\beta\wedge\lambda_M)$. This directly implies that if $\iota_Xd\lambda_{\tilde{M}}=-d\tilde{f}_t=-e^{-g}\pi_g^*(d_\beta f_t)$, then $\pi_g^*\left(\iota_{\pi_{g*}(X)}d_\beta\lambda_M\right)=-\pi_g^*d_\beta f_t$, yielding the commutative diagram.
	
	\begin{Def}
		Given a Hamiltonian isotopy $\phi$ of $(T^*M,\lambda_M,\beta)$ whose Hamiltonian is $f_t$, the Hamiltonian isotopy of $(T^*\tilde{M},\lambda_{\tilde{M}},0)$ whose Hamiltonian is $\tilde{f}_t=e^{-g}f_t\circ\pi_g$ is called the symplectized lift of $\phi$ and will be denoted $\tilde{\Phi}$. Similarly, for any Lagrangian $\Lambda$ in $(T^*M,\lambda_M,\beta)$, $\pi_g^{-1}(\Lambda)$ is called the symplectized lift of $\Lambda$.
	\end{Def}
	
	\begin{Rem}\label{bertelson&co1}
		The symplectized lift of a Hamiltonian isotopy $\phi$, that we here call $\tilde{\Phi}$ in this paper, is the lift noted $\overline{\phi}$ in \cite{bertelson2026nonsqueezingglobalrigidityresults}.
	\end{Rem}

	This implies that we can now use Guillermou, Kashiwara, and Schapira's classical constructions to transform our ``$\lcs$'' Hamiltonian isotopies and exact Lagrangians into ``symplectic'' Hamiltonian isotopies and exact Lagrangians.
	
	\subsection{Classical homogeneous lifts}\label{lifts,} Let us first introduce a bit of notation:
	\begin{Def}
		For a manifold $X$ and a submanifold $Y$, $T^*_YX$ will be the conormal of $Y$ viewed in $T^*X$ and $\dot{T}^*X$ will be $T^*X\backslash T^*_XX$.
		
		Moreover, $\lambda_X$ will be the canonical Liouville form on $T^*X$. If the manifold is explicit, the index will be dropped. Idem for $\omega_X=d\lambda_X$.
	\end{Def}
	
	Guillermou, Kashiwara, and Schapira's constructions revolve around the projection:
	
	\begin{align*}
		\rho_{cl}:\;\left((T^*\tilde{M})\times \dot{T}^*\R,\lambda_{\tilde{M}\times\R},0\right)&\rightarrow(T^*\tilde{M},\lambda_{\tilde{M}},0);\\
		(x,\xi,s,\sigma)&\mapsto \left(x, \frac{\xi}{\sigma}\right)
	\end{align*}
	
	The first use of this projection is to transform exact Lagrangians into conic Lagrangians.
	
	\begin{Def}
		Assume that $\Lambda\overset{i}{\subset} (T^*M,\lambda_M,\beta)$ is an exact Lagrangian submanifold such that $i^*\lambda_M=d_\beta f$. Take $\tilde{M}$ the integral cover of $\beta$ and $g$ a primitive of $\beta$ in the cover. Then 
		\[C^{cl}(\Lambda)=\{(x,\xi,s,\sigma)\in\rho_{cl}^{-1}\circ\pi_g^{-1}(\Lambda)\,:\, s=-e^{-g(x)}f\circ D\pi_g(x,\xi/\sigma)\}\]
		
		is called the classical conification of $\Lambda$ and it is a Lagrangian in $\left((T^*\tilde{M})\times \dot{T}^*\R,\lambda_{\tilde{M}\times\R},0\right)$.
	\end{Def}
	
	\begin{Rem}
		Note that if $\alpha\neq id$ is a deck transformation of $T^*(\tilde{M}\times\R)$, then $\alpha(C^{cl}(\Lambda))\neq C^{cl}(\Lambda)$ in general.
	\end{Rem}
	
	The second use of this projection is to transform Hamiltonian isotopies into homogeneous Hamiltonian isotopies.
	
	\begin{Pro}[\cite{Guillermou2012SQHIQNP}]\label{rhoclass}
		Given a Hamiltonian isotopy $\tilde{\Phi}:T^*\tilde{M}\times I\rightarrow T^*\tilde{M}$ of $(T^*\tilde{M},\lambda_{\tilde{M}},0)$ whose Hamiltonian is $\tilde{f}_t$, then there is a Hamiltonian isotopy $\Phi^{cl}:T^*\tilde{M}\times\dot{T}^*\R\times I\rightarrow T^*\tilde{M}\times\dot{T}^*\R$ in $(T^*\tilde{M}\times\dot{T}^*\R,\lambda_{\tilde{M}\times\R},0)$ such that:

		\begin{enumerate}
			\item its Hamiltonian is $f_t^{cl}=\sigma\tilde{f}_t\circ\rho_{cl}$, with $\sigma$ the coordinate in the cotangent bundle over $\R$;
			\item we have the following commutative diagram: \[\begin{array}{ccc}
				T^*\tilde{M}\times\dot{T}^*\R\times I & \overset{\Phi^{cl}_t}{\rightarrow}& T^*\tilde{M}\times\dot{T}^*\R\\
				\rho_{cl}\times id\downarrow& \circlearrowleft&\downarrow\rho_{cl}\\
				T^*\tilde{M}\times I&\overset{\tilde{\Phi}_t}{\rightarrow}&T^*\tilde{M}
			\end{array}\]
			\item moreover, there is a smooth function $u:T^*\tilde{M}\times I\rightarrow \R$ such that $\Phi^{cl}_t(x,\xi,s,\sigma)=(x',\xi'\sigma,s+u\circ\rho_{cl}(x,\xi,s,\sigma),\sigma)$ where $\xi\in T_x^*\tilde{M}$, $\sigma\in \dot{T}_s^*\R$, and $(x',\xi')=\tilde{\Phi}_t\circ\rho_{cl}(x,\xi,s,\sigma)$;
			\item if $supp(\tilde{\Phi})\cap T_x^*\tilde{M}$ is compact and locally uniformly bounded for all $x\in \tilde{M}$, then $\Phi^{cl}$ can be extended to $\dot{T}^*(\tilde{M}\times\R)$ by $\Phi^{cl}(x,\xi,s,0)=(x,\xi,s+v(t),0)$ for some smooth map $v:I\rightarrow\R$.
		\end{enumerate}
	\end{Pro}
	
	\begin{Def}
		If $\tilde{\Phi}$ is the symplectized lift of $\phi$, then $\Phi^{cl}$ will be called the classical homogeneous lift of $\phi$
	\end{Def}
	
	\begin{Rem}\label{remrem}
	If $supp(\tilde{\Phi})\cap T_x^*\tilde{M}$ is compact for all $x\in \tilde{M}$, then $\tilde{f}_t$ is constant on a neighborhood of infinity (i.e. away from the $0$-section), and the map $v(t)$ depends only on the value of that constant. Notably, if the constant is $0$, then $v(t)=0$.
	\end{Rem}
	
	\begin{Rem}\label{remanaly}
		The point (4) of the theorem stays true if, when $\|\xi\|\rightarrow+\infty$ $\tilde{f}_t(x,\xi)$ converges to a constant $c$ while $d(\tilde{f}_t)_{(x,\xi)}(\xi)$ converges to $0$ locally uniformly. In this case, the isotopy $\Phi^{cl}$ is just continuous. In general, imposing a stronger convergence of the successive differentials of $\tilde{f}_t$ increases the regularity of $\Phi^{cl}$ (and of $v(t)$). This can be exploited to find analytic approximations of any $\tilde{f}$ constant near $\infty$, that we can then associate to smooth homogeneous Hamiltonian isotopies defined on $\dot{T}^*(\tilde{M}\times\R)$ that are analytic away from $\{\sigma=0\}$.
	\end{Rem}
	
	As shown in \cite{Guillermou2012SQHIQNP}, homogeneous Hamiltonian isotopies of $(T^*\tilde{M},\lambda_{\tilde{M}},0)$ can be ``quantized'' in the sense that for any such isotopy $\phi$ there is a sheaf $K$ over $\tilde{M}\times\tilde{M}\times\R$ such that, among other things, $K$ is unique up to isomorphism and, for any sheaf bounded derived $F$ with compact support, $\dot{SS}(K\circ F_{|t})=\phi_t(\dot{SS}(F))$, where $\dot{SS}$ is the microsupport minus the $0$-section.
	
	One of the uses of $\Phi^{cl}$ and its quantization is a sheaf-theoretical proof of Gromov's non-squeezing theorem, as shown in \cite{Guillermou2019SheavesAS}. The second cornerstone of this proof is a quantity, called the displacement energy of a sheaf (see \cite{Guillermou2019SheavesAS} for more details), and denoted $e(F)$. The equivariance of $\Phi^{cl}$ with respect to the translation of the $s$ coordinate is central to the proof as it ultimately implies that $e(F)=e(K^{cl}\circ F)=e(K^{cl}_t\circ F)$ for all $t$.
	
	However, classical homogeneous lifts do not commute with deck transformations. This point will be further expanded upon when comparing the lifts in subsection \ref{compLift}. This makes the quantization of this lift inconvenient to use in conjunction with asymptotic Betti numbers and equivariant sheaves. As such, we need to define a new homogeneous lift that is equivariant under the action of deck transformations. 
	
	\subsection{Equivariant homogeneous lifts} Whenever we want our lifts to be equivariant under deck transformations, then $\rho_{cl}$ needs to be ``twisted''. As such, we define the map:
	\begin{align*}
		\rho_{eq}:\;\left((T^*\tilde{M})\times \dot{T}^*\R,\lambda_{\tilde{M}\times\R},0\right)&\rightarrow(T^*\tilde{M},\lambda_{\tilde{M}},0);\\
		(x,\xi,s,\sigma)&\mapsto \left(x, e^{-g}\left(\frac{\xi}{\sigma}+sdg\right)\right)
	\end{align*}
	
	We can then define the conification:
	\begin{Def}
		Assume that $\Lambda$ is an exact Lagrangian submanifold of $(T^*M,\lambda_M,\beta)$. Take $\tilde{M}$ the integral cover of $\beta$ and $g$ a primitive of $\beta$ in the cover. Then 
		\[C^{eq}(\Lambda)=\{(x,\xi,s,\sigma)\in\rho_{eq}^{-1}\circ\pi_g^{-1}(\Lambda)\,:\, s=-f\circ D\pi(x,\xi/\sigma)\}\]
		
		is called the equivariant conification of $\Lambda$ and it is a Lagrangian in $\left((T^*\tilde{M})\times \dot{T}^*\R,\lambda_{\tilde{M}\times\R},0\right)$.
	\end{Def}
	
	Working with the lifts on the integral cover has the advantage of allowing us to lift Hamiltonian isotopies (of the $\lcs$ kind) to homogeneous Hamiltonian isotopies (of the symplectic kind).
	\begin{Pro}\label{betaificationofhamisotop}
		Let $\phi:T^*M\times I\rightarrow T^*M$ be a Hamiltonian isotopy of $(T^*M,\lambda_M,\beta)$ associated to the Hamiltonian $f:T^*M\times I\rightarrow\mathbb{R}$. Then there is a homogeneous Hamiltonian isotopy of $(T^*\tilde{M}\times\dot{T}^*\R,\lambda_{\tilde{M}\times\R},0)$ called $\Phi^{eq}$ such that:
		
		\begin{enumerate}
			\item its Hamiltonian is $f^{eq}_t=\sigma\times(\pi_g\circ\rho_{eq})^*f_t$.
			\item we have the following commutative diagram: 
			\[\begin{array}{ccc}
				T^*\tilde{M}\times\dot{T}^*\mathbb{R}\times I & \overset{\Phi^{eq}_t}{\rightarrow}& T^*\tilde{M}\times\dot{T}^*\mathbb{R}\\
				\rho_{eq}\times id\downarrow& \circlearrowleft&\downarrow\rho_{eq}\\
				T^*\tilde{M}\times I&\overset{\tilde{\Phi}_t}{\rightarrow}&T^*\tilde{M}
			\end{array}\]
			where $\tilde{\Phi}$ is the symplectized lift of $\phi$.
			\item there is a smooth function $h_t:T^*M\times I\rightarrow\mathbb{R}$  such that:
			\begin{align*}
				& \qquad\qquad\qquad\qquad\Phi^{eq}_t(x,\xi,s,\sigma)=\\
				&\bigg(x',e^{g(x)}\xi'\sigma-(s+e^{g(x)-g(x')}H_t)\sigma dg, e^{g(x')-g(x)}s+H_t,e^{g(x)-g(x')} \sigma\bigg),
			\end{align*}
			with $(x',\xi')=\tilde{\Phi}_{t}\circ\rho_{eq}\;(x,\xi,s,\sigma)$ and $H_t=h_t\circ\pi_g\circ\rho_{eq}(x,\xi,s,\sigma)$.
			\item if $M$ is connected and $f_t$ is a constant $c_t$ outside of a compact set, then $\Phi^{eq}$ can be extended to a homogeneous Hamiltonian isotopy of $(\dot{T}^*(\tilde{M}\times\R),\lambda)$ such that:
			\[ \Phi^{eq}_t(x,\xi,s,0)=(x,\xi,s+w(t),0)\]
			for some smooth $w:\mathbb{R}\rightarrow\mathbb{R}$.
		\end{enumerate}
	\end{Pro}
	
	\begin{proof}
		Consider the following diffeomorphism:
		\begin{align*}
			L:\qquad T^*(\tilde{M}\times\mathbb{R})&\rightarrow T^*(\tilde{M}\times\mathbb{R})\\
			(x,\xi,s,\sigma)&\mapsto (x,\xi+s\sigma dg,e^{-g(x)}s,e^{g(x)}\sigma)
		\end{align*}
		We can verify $L^*\lambda_{\tilde{M}\times\R}=\lambda_{\tilde{M}\times\R}$. Since $f^{eq}=f^{cl}\circ L$, we have the following diagram:
		\[\begin{array}{ccc}
			T^*\tilde{M}\times\dot{T}^*\mathbb{R}\times I & \overset{\Phi^{eq}_t}{\rightarrow}& T^*\tilde{M}\times\dot{T}^*\mathbb{R}\\
			L\times id\downarrow& \circlearrowleft&\downarrow L\\
			T^*\tilde{M}\times\dot{T}^*\mathbb{R}\times I & \overset{\Phi^{cl}_t}{\rightarrow}& T^*\tilde{M}\times\dot{T}^*\mathbb{R}
		\end{array}\] Therefore, we can apply proposition \ref{rhoclass} to get the items 1, 2 and 4 since $\rho_{eq}=\rho_{cl}\circ L$. In particular, this shows that $\Phi^{eq}$ commutes with deck transformations and that
		\begin{align*}
			& \qquad\qquad\qquad\qquad\Phi^{eq}_t(x,\xi,s,\sigma)=\\
			&\bigg(x',e^{g(x)}\xi'\sigma-(s+e^{g(x)}u_t\circ\rho_{eq})\sigma dg, e^{g(x')-g(x)}s+e^{g(x')}u_t\circ\rho_{eq},e^{g(x)-g(x')} \sigma\bigg).
		\end{align*} 
	\end{proof}

	\begin{Def}
		The isotopy $\Phi^{eq}$ will be called the equivariant homogeneous lift of $\phi$.
	\end{Def}
	
	\begin{Rem}
	Mutatis mutandis, remark \ref{remanaly} applies here too.
	\end{Rem}
	
	This proves the first point of theorem \ref{thm2}. Let us then further explore the properties of those equivariant homogeneous lifts. Before proceeding, we will recall the following facts about homogeneous Hamiltonian isotopies:
	
	\begin{The}\label{Quant}(\cite{Guillermou2012SQHIQNP}, proposition 3.2, theorem 3.7)
		Set $X=\tilde{M}\times\R$. Let $\phi_t$ be a smooth homogeneous Hamiltonian isotopy of $(\dot{T}^*X,\lambda,0)$ (with $\phi_0=id$), and take $\mathcal{V}_t$ its Hamiltonian vector field. Define:
		\[\Lambda_\phi=\bigl\{\left(\phi_t(x,\xi),(x,-\xi),(t,-\lambda_{\phi_t(x,\xi)}(\mathcal{V}_t))\right):(x,\xi)\in T^*_x X-\{(x,0)\}\bigr\}.\] Then there exists $K\in D^{lb}(X\times X\times I)$ such that:
		\begin{itemize}
			\item $SS(K)\subset\Lambda_\phi\cup X\times X\times I\subset T^* (X\times X\times I)$.
			\item $K_{|t=0}\simeq \K_{\Delta_{X\times X}}$.
		\end{itemize}
		Such a $K$ is unique up to isomorphism.
	\end{The}
	
	As a direct consequence of the previous theorem, we have:
	
	\begin{Rem}\label{rem37}(theorem 3.7, \cite{Guillermou2012SQHIQNP})
		Let $h_t:T^*M\rightarrow\mathbb{R}$ be a map that is constant outside of a compact set. Let $\phi_t$ be the Hamiltonian isotopy of $(T^*M,\lambda,\beta)$ whose Hamiltonian is $h_t$, and take $\Phi^{eq}$  its equivariant homogeneous lift. Then $\Phi^{eq}$ has a quantization $K^{eq}\in D^{lb}((\tilde{M}\times \R)^2\times I)$.
		
		Moreover, $\Phi^{eq}$ descends to an isotopy $\Phi$ on $T^*({M}\times \R)$ such that $\pi\circ\Phi^{eq}=\Phi\circ\pi$, where $\pi:\tilde{M}\times \R\rightarrow{M}\times \R$ is the canonical projection. This new isotopy can be quantized by a sheaf $K$. By the unicity of the quantization,  \[(\pi\times\pi\times id_I)^{-1} K\simeq\bigoplus_{\alpha\in \mathcal{H}}(\alpha\times id_{\tilde{M}\times\R}\times id_I)_*K^{eq}.\] 
		
		Similarly, $\underline{K}:=R(\pi\times\pi\times id_I)_*K^{eq}$. Since $K^{eq}$ is equivariant with respect to the diagonal action of $\mathcal{H}$, the sheaf $\underline{K}$ comes with a $\mathcal{H}$-action. The unicity of the quantization implies that \[K\simeq R\mathcal{H}om_{\K_{(M\times\R)^2\times I}[\mathcal{H}]}(\K_{(M\times\R)^2\times I},\underline{K}).\]
		
		Since $K$ is a locally bounded sheaf, $K^{eq}_C$ is a bounded sheaf for any $C$ such that $(\pi\circ\pi\circ id_I)(C)$ is compact. In particular, if $F\in D^b(\tilde{M}\times\R)$ is equivariant under deck tranformations, and has a support whose projection on $M\times\R$ is compact, then $K^{eq}\circ F$ is a bounded sheaf.
	\end{Rem}	

We also have the following obvious lemma:
	
	\begin{Lem}\label{lem3.1}
		With the notations of the previous remark, for any given deck transformations $\alpha$ of $\tilde{M}$, we have $(\alpha_1\times\alpha_2)_* K^{eq}\simeq K^{eq}$, with
		\begin{itemize}
			\item $\alpha_1=\alpha\times id\times id$.
			\item $\alpha_2=id\times\alpha\times id$.
		\end{itemize} 
	\end{Lem}
	Remember that the group of deck transformations is a free $\Z$-module.
	\begin{proof}
		Quantization is unique up to isomorphism, and $\Phi^{eq}$ commutes with the deck transformations since $f^{eq}$  is itself equivariant under deck transformations. Therefore, for any given deck transformation $\alpha$, we have $(\alpha_1+\alpha_2)_* K^{eq}\simeq K^{eq}$.
	\end{proof}
	
	The previous remark and lemma ultimately imply that:
	
	\begin{Pro}\label{betakernel} Assume that $I\subset \R$ is a relatively compact open interval centered on $0$.
		Let $F\in D^b(\tilde{M}\times\R)$ be equivariant under deck transformations with compact support up to deck transformations. Take $K^{eq}$ as in the previous remark, then $K^{eq}_J\circ F\in D^{b}(\tilde{M}\times\R\times J)$ is equivariant under deck tranformations. If the projection of $supp(F)$ on $M\times\R$ is compact, then so is the (canonical) projection of $supp(K_J^{eq}\circ F)$ on $M\times\R$.
	\end{Pro}
	\begin{proof}
		Let $q_1:(\tilde{M}\times \R)\times(\tilde{M}\times \R)\times I\rightarrow (\tilde{M}\times\R)\times I$ be the projection on the first and last factors and $q_2:(\tilde{M}\times \R)\times(\tilde{M}\times \R)\times I\rightarrow \tilde{M}\times \R$ be the projection on the second factor. Take $\alpha$, $\alpha_1$, and $\alpha_2$ as in the previous lemma. Then,
		\begin{align*}
			\alpha_{1*}K^{eq}\circ F&=\alpha_{1*}Rq_{1!}(K^{eq}\overset{L}{\otimes}q_2^{-1}F)\simeq Rq_{1!}((\alpha_1+\alpha_2)_* K^{eq}\overset{L}{\otimes}(\alpha_1+\alpha_2)_* q_2^{-1}F)\\
			&\simeq Rq_{1!}( K^{eq}\overset{L}{\otimes} q_2^{-1}\alpha_{*} F)\simeq Rq_{1!}(K^{eq}\overset{L}{\otimes} q_2^{-1}F)=K^{eq}\circ F.
		\end{align*}
		We now have to prove that $K^{eq}_J\circ F$ is bounded and that the projection of its support is compact. We will make use of remark \ref{rem37}. First, note that, since $F_{|W_0}$ has compact support, we may assume that $\Phi_{|J}^{eq}$ has compact support in the $s$ direction. This implies that both $K^{eq}_J$ and $F$ are bounded, and therefore so is their composition. Moreover, $\pi(supp(K_J\circ F))\subset supp(K_J\circ R\pi_*F)$, which is compact.
	\end{proof}
	
	\begin{Rem}
		This proves the second point of theorem \ref{thm2}.
	\end{Rem}

	However, we are not quite done yet with this subsection, as we now wish to understand how pushing with a homogeneous Hamiltonian isotopy modifies a $\beta$-sheaf. First, let us recall the following proposition:
	
	\begin{Pro}(\cite{Guillermou2012SQHIQNP}, assertion 4.4)\label{PropCompositionKernel}
		Given a manifold $X$, let $F\in D^b(X)$ be with compact support and $K$ be the quantization of some homogeneous Hamiltonian isotopy of $(T^*X,\lambda,0)$, then
		\begin{itemize}
			\item $SS(K\circ F)\cap X\times T^*I\subset X\times I$.
			\item $SS((K\circ F)_{t=t_0})\cap (\dot{T}^*X)\subset\phi_{t_0}(SS(F))\cap (\dot{T}^*X)$.
		\end{itemize}
	\end{Pro}
	
	\begin{Cor}
	The previous proposition holds for $K^{eq}$ and $F\in D(\tilde{M}\times\R)$ equivariant under deck transformations such that $supp(F_{W_0\times\R})$ is compact.
	\end{Cor}
	
	\begin{proof}
		There is a homogeneous Hamiltonian isotopy $\Phi$ of  $(\dot{T}^*(M\times\R),\lambda,0)$ such that $\pi\circ\Phi^{eq}=\Phi\circ\pi$, where $\pi$ is the obvious projection. By the existence and unicity of quantization, this isotopy can be quantized by a sheaf $K\simeq R(\pi\times\pi\times id_I)_* K^{eq}$. Moreover $SS(F)=\pi^{-1}SS(R\pi_* F)$ where $R\pi_*F$ has compact support. Reasoning on $SS(K\circ F)$ yields the result.
	\end{proof}
	
	This corollary, together with proposition \ref{MorseFaisceaux}, yields the following result:
	\begin{Cor}\label{Guillermoucorollaire1.7}
		Let $I=(-a,a)\subset\mathbb{R}$. Take $F\in D^{}(\tilde{M}\times\mathbb{R})$ a $\beta$-sheaf. Let $\Phi^{eq}$ be a subanalytic homogeneous Hamiltonian isotopy of $(T^*(\tilde{M}\times\mathbb{R}),\lambda,0)$ that is equivariant w.r.t. deck transformations, and call $K^{eq}$ its quantization.
		
		Then, writing $(K^{eq}\circ F)_b:=(K^{eq}\circ F)_{\{t=b\}}$, we have that, for all $s$ and $t$ in $I$, \[ b_j^\beta ((K^{eq}\circ F)_s)=  b_j^\beta ((K^{eq}\circ F)_t).\] 
	\end{Cor}
	To prove this corollary, first let us state this technical variant of proposition \ref{MorseFaisceaux}:
	\begin{Lem}\label{important} Let $F\in D(\tilde{M}\times\R)$ be a sheaf equivariant under deck transformations such that $supp(F)_{W_0\times\R}$ is compact, and $K^{eq}$ the quantization of an equivariant homogeneous Hamiltonian isotopy.
		
		Take $\psi:\tilde{M}\times\R\times I\rightarrow \R$ a map that is $C^1$-close to the projection on the time variable and equivariant under deck transformations. Define $\Lambda_\psi$ as in the proposition \ref{MorseFaisceaux}. Take $\Lambda_{\pm\psi}=\Lambda_\psi\cup\Lambda_{-\psi}$.
		
		Define $\mathcal{K}^k:=(K^{eq}\circ F)_{W_k\times\R\times I\backslash \partial_-W_k\times\R\times I}$.
		
		Given $t'$ and $t$ in $I$ such that $s<t$, define $\mathcal{K}_+^k:=\mathcal{K}_{W_k\times\R\times \{t'<x\leq t\}}$ and $\mathcal{K}_-^k:=\mathcal{K}_{W_k\times\R\times \{t'\leq x< t\}}$. 
		
		Assume that, for any $k$, the intersection of $\Lambda_{\psi}$ and $SS(\mathcal{K}_+^k)$ is finite. Index the projection of the points $(x^1,\xi^1)\in\Lambda_{\psi}\cap SS(\mathcal{K}_+^k)$ on $\tilde{M}\times\R$  with $I_k$. Write $D^k_i=\{y:\psi(y)\geq\psi(x^1_i)\}$ for every $i\in I_k$.
		
		Assume that, for any $k$, the intersection of $\Lambda_{-\psi}$ and $SS(\mathcal{K}_-^k)$ is finite. Index the projection of the points $(x^2,\xi^2)\in\Lambda_{-\psi}\cap SS(\mathcal{K}_-^k)$ on $\tilde{M}\times\R$ with $J_k$. Write $B^k_i=\{y:\psi(y)\leq\psi(x^2_i)\}$ for every $i\in J_k$.
		
		Assume that:
		\begin{enumerate}
			\item[1.] $\{x^1_i\,:\,i\in I_k\}$ and $\{x^2_i\,:\,i\in J_k\}$ are finite.
			\item[2.] $B^k_i:=R\Gamma_{D^k_i}(\mathcal{K}_+^k)_{|x^1_i}\in D^b(Mod^f(\K))$ for all $i\in I_k$.
			\item[3.] $A^k_i:=R\Gamma_{B^k_i}(\mathcal{K}_-^k)_{|x^2_i}\in D^b(Mod^f(\K))$ for all $i\in J_k$.
		\end{enumerate}
		Then $b_j (((K\circ F)_t)_{W_k\times\R})=b_j (((K\circ F)_{t'})_{W_k\times\R})+O_{+\infty}(\#\partial\mathcal{H}_k),$ both sides being potentially equal to $+\infty$.
	\end{Lem}
	\begin{proof}
		As we saw during the proof of proposition \ref{betakernel}, if the support of $F_{W_k\times\R}$ is compact, then so is the support of $(K\circ F)_{W_k\times\R\times J}$ for any compact $J\subset I$. Moreover, proposition \ref{PropCompositionKernel} implies that, for $p:\tilde{M}\times\R\times I\rightarrow I$ the projection, $\Lambda_p$ (resp. $\Lambda_{-p}$) only intersects the microsupport of $\mathcal{K}_+^k$ (resp. $\mathcal{K}_-^k$) along $(\partial W_k)\times\R$. 
		
		Combining this observation with the microlocal Morse lemma yields \[(-1)^l\sum_{j\leq l}(-1)^jb_j(\mathcal{K}_+^k)\leq (-1)^l\sum_{j\leq l}(-1)^j\left(\sum_i b_j(B^k_i)\right),\] and \[(-1)^l\sum_{j\leq l}(-1)^jb_j(\mathcal{K}_-^k)\leq (-1)^l\sum_{j\leq l}(-1)^j\left(\sum_i b_j(A^k_i)\right),\] where both of the right-hand sides of the inequalities are proportional to the size of the boundary of $W_k$ since $d\phi$ is close to $ds$.
		
		In particular, for any $j$, both $b_j(\mathcal{K}_+^k)$ and $b_j(\mathcal{K}_-^k)$ are at most some $O_{+\infty} ((2k+1)^{r-1})$. Here, the $(2k+1)^{r-1}$ corresponds to the number of points in the boundary of $\mathcal{H}_k$.
		
		Moreover, given a manifold $X$ and $Z'\subset Z$ two closed subsets, we have the exact sequence for any derived sheaf $G$:
		
		\[\ldots\rightarrow G_{Z\backslash Z'}\rightarrow G_{Z}\rightarrow G_{Z'}\overset{+1}{\rightarrow}\ldots\]
		which implies \[b_j(G_{Z})\leq b_j(G_{Z\backslash Z'})+ b_j(G_{Z'})\leq b_j(G_{Z\backslash Z'})+  b_{j+1}(G_{Z\backslash Z'})+ b_j(G_{Z}).\]
		
		Therefore, taking $Z=W_k\times\R\times[t',t]$ and either $Z'=W_k\times\R\times\{t'\}$ or $Z'=W_k\times\R\times\{t\}$ yields, for $t_0=t'$ or $t'=t$:
		\begin{align*}
			{b_j(\mathcal{K}_{Z})}\leq {b_j(\mathcal{K}_{t_0})+O_{+\infty}((2k+1)^{r-1})}
			\leq{b_j(\mathcal{K}_Z)+O_{+\infty}((2k+1)^{r-1})},
		\end{align*}
		with $\mathcal{K}_{t_0}=(K^{eq}\circ F)_{W_k\times\R\times \{t'\}\backslash \partial_-W_k\times\R\times \{t'\}}$.
		 This implies: \begin{align*}
		 	-O_{+\infty}((2k+1)^{r-1}) &\leq b_j \big(((K^{eq}\circ F)_{t})_{W_k\times\R}\big)-b_j \big(((K^{eq}\circ F)_{t'})_{W_k\times\R}\big)\\&\leq O_{+\infty}((2k+1)^{r-1}).
		 \end{align*}
	\end{proof}
	
	\begin{Rem}
	The transversality condition for $\psi$ in the lemma holds for a generic $\psi$ if $SS(R\pi_*F)\subset \bigcup_i T^*X_i$ where $(X_i)_i$ is a stratification of $M\times\R$ by smooth submanifolds.
	\end{Rem}
	
	However, even if $F$ is $\R$-constructible, nothing guarantees that $\Phi^{eq}_t(SS(F))$ lies in the cotangent bundle of a $\mu$-stratification. 
	
	\begin{Rem}
	Assume that $\Phi^{eq}$ is subanalytic. If we also assume that $F$ is  $\R$-constructible, this implies that $\mathcal{K}_+^k$ and $\mathcal{K}_-^k$ are also $\R$-constructible (and, in particular, weakly $\R$-constructible); and both $\mathcal{K}_+^k$ and $\mathcal{K}_-^k$ are with compact supports. Therefore, in this case, the assumptions 2 and 3 of lemma \ref{important} are satisfied, and that, if $\psi$ satisfies the transversality condition in the lemma, then assumption 1 is also verified.
	\end{Rem}
	
	\begin{Rem}\label{rem4}
		If $F$ is a $\beta$-sheaf such that $\dot{SS}(F)\subset\{\sigma\neq0\}$. Then we may weaken the assumption on $\Phi^{eq}$ to only being subanalytic away from $\{\sigma=0\}$, which can be obtained by taking an appropriate approximation of a compactly supported $f_{t}$ and then taking the equivariant homogeneous lift of the resulting subanalytic hamiltonian isotopy. Indeed, one can easily show that $\dot{SS}(K^{eq}\circ F)$ is contained in the cotangent bundle of a $\mu$-stratification by, for example, reasoning on $T^*(M\times\R)$. Then, in light of remark \ref{rem37} and lemma 3.5 of \cite{Guillermou2012SQHIQNP}, one can show that the stalks are perfect.
	\end{Rem}
	
	\begin{proof}[proof of corollary \ref{Guillermoucorollaire1.7}]
		Since both $\mathcal{K}_+^k$ and $\mathcal{K}_-^k$ are $\R$-constructible with compact support, there is a finite family of relatively compact subanalytic submanifolds $(X_a)_{a\in_A}$ such that: \[SS(\mathcal{K}_+^k)\cup SS(\mathcal{K}_-^k)\subset\cup_{a\in A}T_{X_a}W_k\times [-c,c]\times I.\]  Thom's jet transversality theorem implies that there is an open dense set of functions $f\in C^\infty(\tilde{M}\times\R\times I)$ for which $\Lambda_f$ is transverse to each $T_{X_a}W_k\times [-c,c]\times I$. Since taking a finite intersection of open dense sets yields an open dense set, there is a map $\psi$ such that the conditions of the lemma are verified. One may also assume that $\psi$ is periodic w.r.t. deck transformations by reasoning on the base manifold $T^*(M\times\R)$.
	\end{proof}
	
	\begin{Rem}
		Provided we know that $b_j^\beta(F)$ exists, we may weaken the conditions on $F$ to being weakly $\R$-constructible and equivariant under deck transformations.
	\end{Rem}

	\begin{proof}[proof of theorem \ref{thm2}]
		The theorem is proven by putting proposition \ref{betaificationofhamisotop}, proposition \ref{betakernel}, and corollary \ref{Guillermoucorollaire1.7} together.
	\end{proof}
	
	\subsection{Links between the different lifts}\label{compLift} Recall the Liouville diffeomorphism:
	\begin{align*}
		L:\qquad T^*(\tilde{M}\times\mathbb{R})&\rightarrow T^*(\tilde{M}\times\mathbb{R})\\
		(x,\xi,s,\sigma)&\mapsto (x,\xi+s\sigma dg,e^{-g(x)}s,e^{g(x)}\sigma)
	\end{align*}
	Trivially, for any exact Lagrangian $\Lambda$ in $(T^*M,\lambda_M,\beta)$, $L(C^{eq}(\Lambda))=C^{cl}(\Lambda)$. Moreover, as used in the proof of proposition \ref{betaificationofhamisotop}, \[\begin{array}{ccc}
		T^*\tilde{M}\times\dot{T}^*\mathbb{R}\times I & \overset{\Phi^{eq}_t}{\rightarrow}& T^*\tilde{M}\times\dot{T}^*\mathbb{R}\\
		L\times id\downarrow& \circlearrowleft&\downarrow L\\
		T^*\tilde{M}\times\dot{T}^*\mathbb{R}\times I & \overset{\Phi^{cl}_t}{\rightarrow}& T^*\tilde{M}\times\dot{T}^*\mathbb{R}
	\end{array}\]
	Recall the map $T_c:(x,\xi,s,\sigma)\mapsto(x,\xi,s+c,\sigma)$, and define $D_{c}:(x,\xi,s,\sigma)\mapsto (x,\xi-e^gc\sigma dg,s+e^g c,\sigma)$.
	One can easily verify that $L^{-1}\circ T_c= D_{c}\circ L^{-1}$. Finally, recall that $\Phi^{cl}$ commutes with $T_c$. This implies that $\Phi^{eq}\circ D_c= D_c\circ L^{-1} \circ\Phi^{cl}\circ L=D_c\circ \Phi^{eq}$, where $D_c$ doesn't commute with deck transformations, whereas $\Phi^{eq}$ does. 
	On the other hand, $T_c$ commutes with deck transformations, whereas $\Phi^{cl}$ does not. 
	These observations should be considered in light of the notion of displacement energy for sheaves via the Tamarkin morphism, first considered by Tamarkin in \cite{tamarkin2008microlocalconditionnondisplaceablility}. The definition use in this paper will be the one used in \cite{Guillermou2019SheavesAS} (see also section \ref{non-squeezingsymp} of this paper).  
	
	On the sheaf side of things, by propositions 5.4.4 and 5.4.5 of \cite{Kashiwara1990SheavesOM}, we have that for any derived sheaf $F$ on $T^* \tilde{M}\times T^*\R$ we have $SS(f^{-1}F)=L(SS(F))$ where $f:(x,s)\mapsto(x,e^{g(x)}s)$. This last equality implies $(f\times f\times id_I)_*K^{cl}\simeq K^{eq}$. Similarly, we can note that, for $d_c(x,s)=(x,s+e^gc)$, we have $D_{-c}(SS(F))=SS(d^{-1}_{c}F)$.

	
\section{A sheaf-theoretical proof of the Chantraine-Murphy theorem}\label{sectionChantraineMurphy}
	We will now shift our focus to proving the Chantraine-Murphy theorem. This will be split into two main points: first, we will see how some sheaves can precisely keep track when the differential of a map intersects its microsupport. This will then be used to prove the theorem.
	
	\subsection{Simple sheaves and intersections}
	We will now make an observation about simple sheaves (definition 7.5.4 in \cite{Kashiwara1990SheavesOM}) and the ``conification'' of Lagrangians.
	\begin{Lem}\label{LemmefaisceauSimple}
		Let $\Lambda$ be a $\beta$-exact Lagrangian of $T^*M$ defined around some point $x_0\in T^*M$ and let $\beta'$  be the pullback of $\beta$ on $\tilde{M}\times\mathbb{R}$. Let $F\in D^b(\tilde{M}\times\mathbb{R})$, assume that $F$ is simple along the conification $C^{eq}_\beta(\Lambda)$ at $(x_0,s_0)$ for some $s_0\in T_{x_0}(\tilde{M}\times\R)$ and for $x_0$ in the interior of $W_0$.
		Take $g$ a primitive of $\beta'$ in $\tilde{M}\times\mathbb{R}$.
		
		Assume that there is a smooth $f:M\times\mathbb{R}\rightarrow\mathbb{R}$ and $\psi=e^{-g}f$ such that for all $(x,s)\in\tilde{M}\times\mathbb{R}$, we have that $d\psi_{(x,s)}\neq0$, and $\Lambda_\psi$ intersects  the conification $C^{eq}_\beta(\Lambda)$ transversely at $(x_0,s_0)$.
		
		Then, $\sum_j dim(H^j(R\Gamma_{\{(x,s)\in W_0\times\mathbb{R}:\psi(x,s)\geq\psi(x_0,s_0)\}}(F)_{|(x_0,s_0)}))=1$.
	\end{Lem}
	\begin{proof} This is a direct consequence of the definition of a simple sheaf (see lemma 4.3 in \cite{Guillermou2012SQHIQNP}).
	\end{proof}

	\begin{Pro}\label{Propositioncontagefaisceaux}
		Let $\Phi^{eq}$ be an equivariant subanalytic homogeneous Hamiltonian isotopy of $(T^*(\tilde{M}\times\R),\lambda,0)$. Let $\beta'$ be the pullback of $\beta$ to $M\times\mathbb{R}$ and $F_0$ be a $\beta'$-sheaf. Take $g$ a primitive of $\beta'$ in $\tilde{M}\times\mathbb{R}$.
		
		Take $S_0=\dot{SS}(F_0)$ and assume there is a subset $S_{0,reg}\subset S_0$ such that $S_{0,reg}$ is the lift $C^{eq}_\beta(\Lambda)$ of an exact Lagrangian $\Lambda$. Assume $F_0$ is simple along $S_{0,reg}$.
		
		Let $f:M\times\mathbb{R}\rightarrow\mathbb{R}$ and $\psi=e^{-g}f$ be the lift of $f$ to $\tilde{M}\times\mathbb{R}$ such that $d\psi$ is never $0$.
		Take $t_0\in I$ and assume that:
		\begin{enumerate}
			\item $\Lambda_\psi\pitchfork\Phi^{eq}_{t_0}(S_0)\subset \Lambda_\psi\pitchfork\Phi^{eq}_{t_0}(S_{0,reg})$
			\item $\Lambda_\psi\pitchfork\Phi^{eq}_{t_0}(S_0)\cap T^*(W_0\times\mathbb{R})\subset T^*int(W_0\times\mathbb{R})$ where $int(\cdot)$ is the interior.
		\end{enumerate}
		Then, $\#\pi_{g}(\{\Lambda_\psi\pitchfork\Phi^{eq}_{t_0}(S_0)\})\geq \sum_j b_j^\beta (F_0)$.
	\end{Pro}
	\begin{proof}
		Take $F_t:=(K\circ F_0)_{|t}$.
		Since $\Phi^{eq}$ is homogeneous, we have that $F_t$ is simple along $\Phi^{eq}_t(S_{0,reg})$ and, according to corollary \ref{Guillermoucorollaire1.7}, $ b_j^\beta (F_t)= b_j^\beta (F_0)$.
		Thanks to lemma \ref{LemmefaisceauSimple}, we can give a lower bound to the number of intersection points since:\[\sum_{z_0\in\Lambda_\psi\pitchfork\Phi^{eq}_{t_0}(S_0)\cap T^*W_0}\sum_j dim(H^j(R\Gamma_{\{z\in W_0\times\mathbb{R}:\psi(z)\geq\psi(z_0)\}}(F_{t_0})_{|z_0}))\]\[=\#\Lambda_\psi\pitchfork\Phi^{eq}_{t_0}(S_0)\cap T^* (W_0\times\mathbb{R}).\]
		 Moreover, direct application of proposition \ref{MorseFaisceaux} yield: \[\sum_{z_0\in\Lambda_\psi\pitchfork\Phi^{eq}_{t_0}(S_0)\cap T^*( W_0\times\mathbb{R})} dim(H^j(R\Gamma_{\{z:\psi(z)\geq\psi(z_0)\}}(F_{t_0})_{|z_0}))\geq  b_j^\beta (F_0).\]
		 
		 Finally, from the second hypothesis, \[\#\Lambda_\psi\pitchfork\Phi^{eq}_{t_0}(S_0)\cap T^* (W_0\times\mathbb{R})=\#\pi_g(\{\Lambda_\psi\pitchfork\Phi^{eq}_{t_0}(S_0)\}).\]
		 Finally, one can note that from a fundamental domain to its deck transformation, $\Lambda_\psi$ scales by $e^{-g}$.
	\end{proof}
	\subsection{The Chantraine-Murphy theorem}\label{thechantraine}
	Let us begin the proof of theorem \ref{thm1} by giving a slightly more precise statement.
	\begin{The}(\cite{Chantraine2016ConformalSG})
		Let $M$ be a (non-empty) closed manifold, $\beta\in\Omega^1(M)$ be closed and
		$\Phi$ be a Hamiltonian isotopy of $(T^*M,\lambda,\beta)$ with compact support, where $\lambda$ is the canonical Liouville form.
		
		Let $c=\sum_jdim(HN_j(M,\beta))$ and assume $\Phi_t(T^*_M(M))$ intersects $T^*_MM$ transversely.
		
		Then, for every $t\in I$ for which $\Phi_t(T^*_M(M))$ and $T^*_MM$ are transverse, \[\#\Phi_t(T^*_M(M))\pitchfork T^*_MM\geq c.\]
	\end{The}
	\begin{proof}\label{proof} 
		In $M\times\mathbb{R}$, identify $M$ and $M\times\{0\}$, then call $\beta$ the pullback of $\beta$ to $M\times\mathbb{R}$. We may assume that $\Phi^{eq}$ is subanalytic: take $F_0=\K_{\tilde{M}\times\{0\}}\in D^b_{\{\sigma\neq 0\}}(\tilde{M}\times\R)$ and notice that $F_0$ is a  $\beta$-sheaf. First, as shown in \cite{currier2025morsenovikovhomologybetacriticalpoints}, note that we have the following inequality for every $j\in\N$: \[ b_j^\beta (F)\geq rank(HN_j(M,\beta)).\]  Therefore, we are left to prove: \[\#\Phi_t(T^*_M(M))\pitchfork T^*_MM\geq \sum_j  b_j^\beta (F).\] Let us also point out that $F_0$ is a simple sheaf. Let $\Phi^{eq}$ be the equivariant homogeneous lift of $\Phi$ and $g$ be a primitive of $\beta$ on $\tilde{M}\times\mathbb{R}$. Take $K$ the quantization of $\Phi^{eq}_{t}$ and $F_{-t}:=K\circ F_{|t}$. Since $F_t$ is also a $\beta$-sheaf, we can select $ W_0$ to satisfy the conditions of the proposition \ref{Propositioncontagefaisceaux}.
		
		Let $\psi(x,s)=e^{-g(x)}s$ be a map from $\tilde{M}\times\mathbb{R}$ to $\mathbb{R}$, then it satisfies $d\psi_{(x,s)}=e^{-g(x)}ds-e^{-g(x)}sdg_x$.
		
		Moreover, explicit computations yield:
		\begin{align*}
			& \qquad\qquad\qquad\qquad\qquad\Phi^{eq}_t(T_{\tilde{M}\times\{0\}}\tilde{M}\times\mathbb{R})=\\
			&\biggl\{\bigg(x_1,e^{g(x_0)}\xi_1\sigma_0-e^{g(x_0)-g(x_1)}H_t\sigma_0 dg, H_t,e^{g(x_0)-g(x_1)} \sigma_0\bigg)\,:\\
			&\,(x_1,\xi_1)=\tilde{\Phi}_{t}(x_0,e^{-g(x_0)}\xi_0/\sigma_0), (x_0,\xi_0,s_0,\sigma_0)\in T_{\tilde{M}\times\{0\}}\tilde{M}\times\mathbb{R}\biggr\},
		\end{align*}
		where $\tilde{\Phi}$ is the symplectized lift.
		
		Therefore, intersections occur precisely when
		\begin{enumerate}
			\item $e^{-g(x)}=e^{g(x_0)-g(x_1)} \sigma_0,$
			\item $s=H_t,$
			\item $-e^{-g(x)}sdg=e^{g(x_0)}\xi_1\sigma_0-e^{g(x_0)-g(x_1)}H_t\sigma_0 dg$
		\end{enumerate} 
		Those three equalities directly imply that $-e^{-g(x)}sdg=e^{g(x_0)}\xi_1\sigma_0-e^{-g(x)}sdg$ and therefore $\xi_1=0$.
		 In other words, for each intersection of $\Phi^{eq}_t(T_{\tilde{M}\times\{0\}}\tilde{M}\times\mathbb{R})$ and $\Lambda_\psi$, there is an intersection of $\Phi_t(M)$ and $M$.
		
		Conversely, if $\tilde{\Phi}_t(x_0,\xi_0)=(x,0)$, then for every $\sigma$ we have the equality $\phi^{-1}(\tilde{\Phi}_{-t}(x_0,\xi_0),0,\sigma)=(x,0,0,\sigma)$, and we can find some $\sigma$ such that $e^{-g}\sigma v_t\circ\pi_g\circ\rho_{eq}=e^{-g}$. That is to say that for each intersection of $\Phi_t(M)$ and $M$, there is an intersection of $\Phi^{eq}_t(T_{\tilde{M}\times\{0\}}\tilde{M}\times\mathbb{R})$ and $\Lambda_\psi$.
		
		Therefore, $\Phi^{eq}_t(T_{\tilde{M}\times\{0\}}\tilde{M}\times\mathbb{R})$ has precisely as many (transverse) intersections with $\Lambda_\psi$ as $\tilde{\Phi}_{-t}(T_{\tilde{M}}\tilde{M})$ has (transverse) intersections with the $0$-section. 
		
		Thanks to proposition \ref{Propositioncontagefaisceaux}, we can conclude that since the intersection is transverse,  \[\#\Phi_t(T^*_M(M))\cap T^*_MM\geq \sum_j  b_j^\beta (F)\geq c.\]
	\end{proof}
	
	It should be noted that, in \cite{Guillermou2012SQHIQNP}, instead of using the map $\psi$, the map used is $proj_s$, the projection on the $s$ variable. We can observe that $\psi=f^{-1}\circ proj_s$, where $f(x,s)=(x,e^{g(x)}s)$ is the map defined in subsection \ref{compLift}.
	
\section{Non-squeezing theorem for symplectic geometry}\label{non-squeezingsymp}
The purpose of this section is to provide a summary of the arguments and constructions laid out in \cite{Guillermou2019SheavesAS}. None of the results of this section are the author's own. That being said, the generalization to the $\lcs$ setting will heavily rely on those results.

\subsection{Displacement energy}
Let us start with a succinct description of the displacement energy. Given a manifold $M$, and the morphisms $s,q_1: M\times\R^2\rightarrow M\times\R$ and $q_2: M\times\R^2\rightarrow \R$ where $s(x,t_1,t_2)=(x,t_1+t_2)$, $q_1(x,t_1,t_2)=(x,t_1)$ and $q_2(x,t_1,t_2)=t_2$. Then, for any $c\geq 0$, we have the map 
\begin{align*}
	P_{c}: D(M\times R) &\rightarrow D_{\{t\geq 0\}}(M\times R)\\
	F&\mapsto Rs_*\left(q_1^{-1}(F)\otimes q_2^{-1}(\mathbb{K}_{[c;+\infty[})\right)
\end{align*}

Consider that the restriction morphism $\mathbb{K}_{[0;+\infty[}\rightarrow \mathbb{K}_{[c;+\infty[}$ induces a morphism $\tau_c: P_{0}\rightarrow P_{c}$, called the Tamarkin morphism.

One should note that $T_{c*}P_0(F)\simeq P_c(F)$ and, if $F\in D_{\{t\geq 0\}}(M\times R)$, then $P_{0}(F) \simeq (F)$. In other words, for sheaves in $D_{\{t\geq 0\}}(M\times R)$, the Tamarkin morphism $\tau_c(F)$ is a morphism $F\rightarrow T_{c*}F$.

We can now define the displacement energy of a sheaf $F$ as \[e(F)=sup\{c\geq0\;|\;\tau_c(F)\neq 0\}.\]

In classical symplectic geometry, the homogeneisation of a Hamiltonian isotopy commutes with $T_c$, which ultimately allows us to conclude that, if $F\in D_{\{\sigma\geq 0\}}(M\times\R)$, $e(F)=e(K_t^{cl}\circ F)$ for all $t$.

\subsection{Parametric non-squeezing for sheaves} The core of the various non-squeezing theorems proven via sheaf theory is the following theorem:

\begin{The}\label{guillermou}(\cite{Guillermou2019SheavesAS})
	Let $r>0$. Then, for any $F\in D(\R^m\times\R^n)$ such that $SS(F)\subset T^*\R^m\times (\R^n\times\{(\xi_1,\ldots\xi_n)\;:\; r\xi_2\geq|\xi_1|\})$ and $supp(F)\subset \R^m\times (]-r;r[\times\R^{n-1})$, we have $e(F)\leq 4r^2$.
\end{The}

This theorem, combined with the invariance of the displacement energy, yields that for any sheaf $F\in D_{\{\sigma\geq 0\}}(\R^m\times\R^n\times\R)$ satisfying the conditions, one cannot have both $e(F)>4r^2$ and $SS(K^{cl}_t\circ F)$ fulfilling the conditions of the theorem, where $K^{cl}$ is the (classical) quantization of some Hamiltonian isotopy of $\R^m\times\R^n$.  

\subsection{Non-squeezing in symplectic geometry}\label{nonsqueezingsympSSEc} For any $x\in\R^n$, take $f_1(x)=\int_0^{\|x\|}\sqrt{1-u^2}du$, $f_2=\pi/2-f_1$ and $V=\{(x,s)\in\R^n\times\R\; :\; f_1(x)\leq s< f_2(x)\}$. Then,
\begin{align*}
	\rho_{cl}(\dot{SS}(\K_V))=&\left\{(x,\xi)\, :\,  \xi\in Vect(x), \|x\|^2+\|\xi\|^2=1\right\} \\ &\cup \{(\underline{0},\xi)\, :\, \|\xi\|\leq 1\}.
\end{align*} 
Since $R\mathcal{H}om(\K_V,T_{c*}\K_V)\simeq \K_{\overline{V\cap T_c V} }$, we can already deduce that $e(\K_V)=\pi/2$. The microsupport can be further simplified by removing the singularities at $(\underline{0},0)$ and $(\underline{0},\pi/2)$. 

We will now proceed to describe how to fix the singularity at $(\underline{0},\pi/2)$. Take $\phi$ a diffeomorphism from a neighborhood of $(\underline{0},\pi/2)\in\R^{n}\times\R$ to a neighborhood of $\underline{0}\in\R^{n+1}$ such that $\phi(V)=\{(x,s)\; :\; \|x\|<s\}$ and $\phi(T_{\pi/2}V)=\{(x,s)\; :\; \|x\|\leq -s\}$. Note that $\Psi_t(x,\xi)=(x+t\xi/\|\xi\|,\xi)$ is a homogeneous Hamiltonian isotopy associated to $h(x,\xi)=\|\xi\|$. As such, it has a quantization sheaf denoted $K_\Psi\in D^{lb}(\R^{2n}\times\R)$. Then $\phi^{-1}(K_{\psi|\R^n\times\{\underline{0}\}\times\R})[-n]$ extends to a sheaf $F$ which fits into the sequence \[\ldots\rightarrow \K_V\overset{\phi[-n]}{\longrightarrow} F\overset{\phi^{-1}}{\longrightarrow}\K_{T_{\pi/2}V}[-n]\overset{+1}{\rightarrow}\ldots\]
A simple computation yields that, for $U$ a small enough neighborhood of $0$, \[\rho_{cl}(\dot{SS}(F_{|U}))\subset\left\{(x,\xi)\; :\; \xi\in Vect(x), \|x\|^2+\|\xi\|^2=1\right\}.\]
However, $F$ still has too big a microsupport at $(\underline{0},0)$ and, now, a new one at $(\underline{0},\pi)$. We can then repeat this process iteratively to find a sheaf $G_0\in D_{\{\sigma\geq 0\}}(\R^n\times\R)$ such that \[G_{0|\R^n\times ]k\pi/2\,,\;(k+2)\pi/2[)}= \left((T_{k\pi/2} )_*F[-kn]\right)_{|\R^n\times ]k\pi/2\,,\;(k+2)\pi/2[}.\]
This sheaf verifies $e(G_0)\geq \pi/2$, which would be enough to prove the following:
\begin{Pro}(\cite{Guillermou2019SheavesAS})
	Let $r<1/\sqrt{2}$ and $\overline{D}_r\subset \R^2$ be the closed disc of radius $r$. There is no Hamiltonian isotopy $\phi$ of $(T^*\R^n,\lambda_{\R^n},0)$ such that \[\phi_1\left(\left\{(x,\xi)\; :\; \xi\in Vect(x), \|x\|^2+\|\xi\|^2=1\right\}\right)\subset \overline{D}_r\times\R^{2n-2}.\]
\end{Pro}

Indeed, first, we need to observe that there is a Hamiltonian isotopy $\psi$ such that $\psi(\overline{D}_r\times\R^{2n-2})\subset ]-a,a[^2\times\R^{2n-2}$ for some $a$ such that $(2a)^2<\pi/2$. Therefore, for $\Phi$ the classical homogeneous lift of $\phi$, we may assume that \[\rho_{cl}(\Phi_1(\dot{SS}(G_0)))\subset ]-a,a[^2\times\R^{2n-2}.\] Since, for $x_1$ big enough, $(G_0)_{(x_1,\ldots,x_n,s)}\simeq0$, the same must hold for $K^{cl}_t\circ G_0$, and therefore $\rho_{cl}(\dot{SS}(K_t\circ G_0))\subset ]-a;a[^2\times\R^{2n-2}$ implies that $K^{cl}_t\circ G_0$ fulfills the conditions of the theorem \ref{guillermou} for $r=a$. In turn, this implies that $e(G_0)\leq 4a^2<\pi/2$, which is a contradiction. 

To get this proposition for $r<1$, we need to build a sheaf $G$ such that $e(G)\geq\pi$. Consider the map $(x,\xi)\mapsto x+i\xi$, giving a correspondence $T^*\R^n\simeq\C^n$. Under this correspondence, the image of the map $(s,\xi)\in(\R/\Z)\times \Sph_1^{n-1}\mapsto e^{2\pi si}\cdot (0,\xi)$ is \[L_0:=\left\{(x,\xi)\; :\; \xi\in Vect(x), \|x\|^2+\|\xi\|^2=1\right\},\] with the convention $Vect(0)=\R^n$. This map can be extended to 
\begin{align*}
j_x:\Sph^1\times\Sph_r(T_x^*\R^n)&\rightarrow (T^*\R^n)^2 \\ (s,\xi,x)&\mapsto \big( (x,-\xi),e^{2\pi si}\cdot (x,\xi)\big),
\end{align*}
 for $r=\sqrt{1-\|x\|^2}$. Viewing $x\in \mathbb{B}_1(\R^n)$ as a parameter, $j_x$ defines an isotopy, also called $j$, of isotropic embeddings of $Im(j_0)$. One should note that $Im(j_0)$ projects back down to $L_0\subset \{0\}\times T^*\R^n$, and more generally, the projection on the second factor defines a family of Lagrangians $L_y$. Given a loop in $L$, the integral of the canonical Liouville form $\lambda$ along this loop is $\pi$ and therefore, $L$ can be lifted to a $\pi$-periodic Legendrian $L'$ in the contactisation $J^1\R^{2n}$ and each $L_y$ can be lifted to some $L'_y$ (viewed as a submanifold of $L'$). In turn, $L'$ can be lifted to $L''$ a conic Lagrangian submanifold of $\dot{T}^*(\R^{2n}\times\R_s)\cap\{\sigma>0\}$ where $T^*\R_s=\R_s\times\R_\sigma$. Moreover, the submanifolds $L'_y$ can be viewed as Legendrians in $J^1\R^n$ with a contact isotopy $\psi$ parametrised by $\mathbb{B}_1(\R^n)$ such that $\psi_y(L_0')=L_y'$. This map can be lifted to a homogeneous hamiltonian isotopy $\Psi$ that can then be quantized by a sheaf $K$ since the parameter space does not matter in the quantization process. And, finally, we can define $G=K\circ G_0\subset D_{L''}(\mathbb{B}_1(\R^n)\times\R^n\times\R)$. By construction, $\rho_{cl}(\dot{SS}(G_{|\{y\}\times\R^n}))\subset L_y$ for any $y$ in the unit ball of $\R^n$ and a more in-depth study of $G$ shows that, indeed, $e(G)\geq \pi$.

\section{Non-squeezing theorem for $\lcs$ geometry}\label{non-squeezing}

As noted in remark \ref{bertelson&co1}, both \cite{bertelson2026nonsqueezingglobalrigidityresults} and this paper begin the study of ``$\lcs$'' Hamiltonian isotopies $\phi$ of $(T^*M,\lambda_M,\beta)$ by taking their symplectized lifts $\tilde{\Phi}$ in $(T^*\tilde{M},\lambda_{\tilde{M}},0)$. However, where the authors of \cite{bertelson2026nonsqueezingglobalrigidityresults} proceed to study $gr_{\lcs}(\phi)$, the projection of the graph of $\tilde{\Phi}_t$ in $T^*{M}\boxtimes T^*M$, we instead elect to quantize the homogeneous equivariant lift. This approach still allows us to recover $gr_{\lcs}(\phi_t)$ via the following remark:
 
\begin{Rem}\label{bertelson&co}
	Recall the definition of $\Lambda_{\Phi^{eq}}$ in theorem \ref{Quant}, and define \[\Lambda_{\Phi^{eq}_t}= \Lambda_{\Phi^{eq}}\cap T^*\left((\tilde{M}\times\R)^2\times\{t\}\right).\]
	We can note that $\Lambda_{\Phi^{eq}_t}$ is an exact conic Lagrangian submanifold of $T^*\left((\tilde{M}\times\R)^2\right)$. Note $\tilde{\pi}_g(x,\xi):=(x,e^{g(x)}\xi)$ the lift of $\pi_g$ to the cover and define the symmetry: \[\text{sym}:(x_1,\xi_1,x_2,\xi_2)\in T^*M_1\times T^*M_2\mapsto(x_2,\xi_2,x_1,\xi_1)\in T^*M_2\times T^*M_1.\] Then, taking the definition of $gr_{\text{lcs}}$ from \cite{bertelson2026nonsqueezingglobalrigidityresults}, \[\left[\text{sym}\circ\left(\tilde{\pi}_g\times{\pi}_g\right)\circ\left(\rho_{eq}\times\rho_{eq}\right)\left(\Lambda_{\Phi^{eq}_t}\right)\right]=gr_{\text{lcs}}(\phi_t),\]
	where the bracket denotes the class in $M$, as defined in \cite{chantraine2024productslocallyconformalsymplectic}
\end{Rem}

It should be noted that $gr_{\lcs}(\phi_t)$ can be defined for $\lcs$ Hamiltonian isotopies of any manifold (as opposed to only cotangent bundles), although spectral selectors are only defined (for now) for compactly supported Hamiltonian diffeomorphisms of $\Sph^1\times\R^{2n}\times\Sph^1$ and $\Sph^1\times\R^{2n+1}$. As we will see in subsection \ref{invariance}, those can be studied as Hamiltonian isotopies of $T^*(\Sph^1\times\R^{n})$.

Nevertheless, the unique value proposition of $\Phi^{eq}$ to study non-squeezing in $\lcs$ geometry is that we can leverage both the classical results for homogeneous ``symplectic'' Hamiltonian isotopies, and the equivariance.

\subsection{Invariances and equivariances of the homogeneous lifts}\label{invariance}
The generalization of the non-squeezing to the $\lcs$ setting necessitates additional considerations. Indeed, for $f:(x,s)\mapsto (x,e^gs)$ and $F$ a $\beta$-sheaf,  $e(f^{-1}F)$ is either $0$ or $+\infty$ if $\beta$ is non-exact. On the other hand, the equivariant lift does not commute with $T_c$ and, therefore, $e(K_t^{eq}\circ F)$ depends on $t$ in general. Nevertheless, considering the specific equivariances of some isotopies can yield some results.

Let $f_t$ be a compactly supported Hamiltonian on $(\R^{2n}\times\Sph^1_\zeta\times\Sph^1_z,\lambda_{\R^n}+d\zeta,dz)$. By slight abuse of notation, we will also call $f_t$ its lift to $T^*(\R^n\times\Sph^1_z)$. We can observe that $d_{dz}(\lambda_{\R^n}+d\zeta)=d_{dz}\lambda_{\R^n\times\Sph^1_z}$. Call $\Phi$ the Hamiltonian isotopy of $(T^*(\R^n\times\Sph^1_z),\lambda_{\R^n\times\Sph^1_z},dz)$ associated to $f_t$. Following proposition \ref{betaificationofhamisotop}, this isotopy can be made into an equivariant Hamiltonian isotopy $\Phi^{eq}$ that is not defined on $\{\sigma=0\}$. Nevertheless, this isotopy possesses as suitable quantization $K^{eq}$ (see the \hyperref[annex]{Annex}).

Take note of the obvious invariance $f_t(x,\xi,z,\zeta+1)=f_t(x,\xi,z,\zeta)$. First, set: \begin{enumerate}
	\item $\Phi_t^{cl}(x,\xi,z,\zeta,s,\sigma)=(x^{cl},\xi^{cl},z^{cl},\zeta^{cl},s^{cl},\sigma^{cl})$;
	\item $\Phi_t^{eq}(x,\xi,z,\zeta,s,\sigma)=(x^{eq},\xi^{eq},z^{eq},\zeta^{eq},s^{eq},\sigma^{eq})$.
\end{enumerate} Recall the trivial equivariances, for any constant $c$: \begin{align*}
&\Phi_t^{cl}(x,\xi,z,\zeta,s+c,\sigma)=(x^{cl},\xi^{cl},z^{cl},\zeta^{cl},s^{cl}+c,\sigma^{cl})\\
\text{and }&\Phi_t^{eq}(x,\xi,z+1,\zeta,s,\sigma)=(x^{eq},\xi^{eq},z^{eq}+1,\zeta^{eq},s^{eq},\sigma^{eq}).
\end{align*}  Since $ -\partial_t\Phi_t^{eq}(x,\xi,z,\zeta,s+1,\sigma)$ is the $\omega$-dual of \[ d(\sigma f_t(x,\frac{\xi}{\sigma},t+\Z,\frac{\zeta}{\sigma}+(s+1)))=d(\sigma f_t(x,\frac{\xi}{\sigma},t+\Z,\frac{\zeta}{\sigma}+s)),\] we deduce that $\Phi_t^{eq}(x,\xi,z,\zeta,s+1,\sigma)=(x^{eq},\xi^{eq},z^{eq},\zeta^{eq},s^{eq}+1,\sigma^{eq})$.

Here, we should note that, for any $k\in\N$ and any $r\geq k$, the Tamarkin morphism $F\rightarrow T_{r*}F$ factors through $F\rightarrow T_{k*}F$. Following the discussion of subsection \ref{compLift}, theorem \ref{theor}, remarks \ref{duh} and \ref{duh2}, and \cite{Guillermou2019SheavesAS} subsection 6.2.1. and cor. 2.1.5, we conclude that:

\begin{Lem}\label{ref}
	Let $\phi$ be a Hamiltonian isotopy of $(T^*(\R^n\times\Sph_z^1),\lambda_{\R^n\times\Sph_z^1},dz)$ with a support whose projection on $T^*\R^n$ is compact, and $F\in D_{\{\sigma>0\}}(\R^n\times\R\times\R)$ be equivariant under deck transformations, with a microsupport that is compact in the $x$ coordinates and periodic in the $s$ coordinates. 
	Assume that $\phi$ commutes with the translation $(x,\xi,z,\zeta)\mapsto (x,\xi,z,\zeta+1)$. Then, for any $k\in\N$ and any $t$, \[e(F)> k\implies e(K^{eq}_t\circ F)\geq k.\]
\end{Lem}

Note that the sheaf $G$ defined in subsection \ref{nonsqueezingsympSSEc} can be straightforwardly pulled back to a sheaf on $\R^{2n}\times\R\times\R$ that fulfills the conditions of lemma \ref{ref}. Moreover, one can easily verify that, if $V$ is scaled by a factor $R_1$ (that is to say, we redefine $f_1(x)=\int_0^{\|x\|}\sqrt{R_1^2-u^2}du$ and $f_2=R_1^2\pi/2-f_1$), then $e(G)\geq \pi R_1^2$. Finally, recall that, for $R_2>0$ and any $r>R_2$, there is a Hamiltonian isotopy $\psi$ such that $\psi(\overline{D}_{R_2}\times\R^{2n-2})\subset]-a,a[^2\times\R^{2n-2}$,  for some $a$ such that $4a^2<\pi r^2$.

Those considerations come together to show the following:

\begin{Lem}
Let $R_1$ and $R_2$ be such that, for some $k\in \N$, $\pi R^2_1\geq k\geq \pi R^2_2$. Then there is Hamiltonian isotopy $\phi$ of $(T^*(\R^{n}\times \Sph^1_z),\lambda_{\R^n\times\Sph^1_z},dz)$ such that $\phi$ commutes with the translation $(x,\xi,z,\zeta)\mapsto(x,\xi,z,\zeta+1)$ such that the projection of $supp(\phi)$ on $T^*\R^{n}$ is compact and \[\phi\left(\overline{\mathbb{B}_{R_1}^{2n}\times T^*\Sph^1_z}\right)\subset D_{R_2}\times\R^{2n-2}\times T^*\Sph^1_z.\]
\end{Lem}
\begin{proof}
If $\pi R^2_1> k$ then the result can be directly derived from the previous discussion around lemma \ref{ref} by noting that $\phi\left(\overline{\mathbb{B}_{R_1}^{2n}\times T^*\Sph^1_z}\right)\subset D_{R_2}\times\R^{2n-2}\times T^*\Sph^1_z$ implies that there is some $\pi R^2_3<\pi R^2_2$ such that $\phi\left(\overline{\mathbb{B}_{R_1}^{2n}\times T^*\Sph^1_z}\right)\subset \overline{D}_{R_3}\times\R^{2n-2}\times T^*\Sph^1_z$. If $\pi R^2_1= k$, then $\phi\left(\overline{\mathbb{B}_{R_1}^{2n}\times T^*\Sph^1_z}\right)\subset D_{R_2}\times\R^{2n-2}\times T^*\Sph^1_z$ implies that there is some $\pi R^2_4> k$ such that $\phi\left(\overline{\mathbb{B}_{R_4}^{2n}\times T^*\Sph^1_z}\right)\subset D_{R_2}\times\R^{2n-2}\times T^*\Sph^1_z$, which is absurd in light of lemma \ref{ref}. 
\end{proof}
Together with the observations made at the beginning of the section, this shows theorem \ref{thm3}.

 \renewcommand{\thesection}{A}

\section{Annex}\label{annex}
In this annex, we properly define the quantization of the Hamiltonian isotopy $\Phi^{eq}$ of section \ref{non-squeezing} by a (derived) sheaf $K^{eq}$, and prove:

\begin{The}\label{theor}
	Let $A\subset\{\sigma>0\}$ be any closed conic subset of $\dot{T}^*(\R^n\times\R\times\R)$ that is periodic in the $s$ coordinate and in the $z$ coordinate, and bounded in the  $x$ coordinate. The quantization $K^{eq}$ induces an equivalence of category:
	\[K_t^{eq}\circ\underline{\quad}:D_{A}(\R^n\times\R\times\R)\tilde{\;\rightarrow} D_{\Phi^{eq}_t(A)}(\R^n\times\R\times\R),\]
	for any $t\in\R$.
	
	Moreover, for $F,G\in D_A(\R^n\times\R\times\R) $,
	\[\text{Hom}(F,G)\overset{\sim}{\leftarrow}\text{Hom}(K^{eq}\circ F,K^{eq}\circ G)\overset{\sim}{\rightarrow}\text{Hom}(K_t^{eq}\circ F,K_t^{eq}\circ G).\]
\end{The}

To prove this theorem, we will need a couple of preliminary results.\newline

\paragraph{\bf{Definition of the quantization}} One should first note that the main statement in \cite{Guillermou2012SQHIQNP} is slightly more general than the version previously stated in this paper, and can be summed up as followed:

\begin{The}[\cite{Guillermou2012SQHIQNP}]
	Let $W$ be an open contractible subset of $\R^n$ containing $0$ and $\Phi^{cl}:\dot{T}^*M\times W\rightarrow \dot{T}^*M$ be a homogeneous (symplectic) Hamiltonian isotopy. Then there exists a unique simple $K^{cl}\in D(M^2\times W)$ such that $\dot{SS}(K^{cl})$ is the graph of $\Phi^{cl}$ and $K^{cl}_{|M^2\times\{0\}}\simeq \K_{\Delta_{M^2}}$.
\end{The}

Let $f_t$ be a non-compactly supported Hamiltonian of $T^*(\R^n\times\Sph^1)$ as above, and 
$\chi_{a}:\R\rightarrow \R$ be a smooth step function that is $1$ for  $|\zeta|<e^\frac{1}{1-a}$ and $0$ for $|\zeta|>e^\frac{1}{1-a}+1$. Then the Hamiltonian isotopy $\Phi'$ associated to $f_t\times\chi_{a}$ can be made into an equivariant homogeneous Hamiltonian isotopy $\tilde{\Phi}': \dot{T}^*\big(\R^n\times\R\times\R\big)\times \big(I\times(-\infty,1)\big)\rightarrow \dot{T}^*\big(\R^n\times\R\times\R\big)$, with $a$ as the second variable, and lifted to equivariant Hamiltonian isotopies $\Phi'^{eq}$, which quantized by $\tilde{K}^{eq}\in D^{lb}(\big(\R^n\times\R\times\R\big)^2\times \big(I\times(-\infty,1)\big))$. The inclusions $j:(-\infty,1)\hookrightarrow (-\infty,1]$ and $j':\{1\}\hookrightarrow (-\infty,1]$ extend to the ambient space, which leads to the following definition:

\begin{Def}\label{def}
	The quantization of $\Phi^{eq}$ will be \[K^{eq}:=j'^{-1}Rj_*\tilde{K}^{eq}\in D\left(\big(\R^n\times\R\times\R\big)^2\times I\right).\]
\end{Def}

No  claim is made about the local boundedness of $K^{eq}$. While most of the results in the various works cited are stated for bounded derived category, they still apply to the unbounded category.\newline

   \paragraph{\bf{The microsupport of $K^{eq}$}}By \cite{Guillermou2022TheSS} prop. 3.4 and \cite{Kashiwara1990SheavesOM} proposition 6.3.2, $\dot{SS}(K^{eq})\cap T^*\big(\R^n\times\R\big)^2\times\big(\dot{T}^*\R\big)^2\times T^*I$ is the graph of the (naive) equivariant homogeneization of the Hamiltonian isotopy associated to $f_t$. We will proceed to give a few more properties of $\dot{SS}(K^{eq})$. 

 Let us write $\tilde{\Phi}_{t,a}^{eq}(x,\xi,z,\zeta,s,\sigma)=:(x_{t,a},\xi_{t,a},z_{t,a},\zeta_{t,a},s_{t,a},\sigma_{t,a})$.  First, note that $f_{t,a}=\sigma\times(f_t\chi_a)\circ(\pi_g\circ\rho_{eq})$ is uniformly Lipschitz in the $\sigma$ variable for any norm. Therefore, by theorem \ref{Quant}, we conclude that $\dot{SS}(\tilde{K}^{eq})\subset\{|\tau|\leq C|\sigma_{t,a}|\}$.
 
 Similarly,  in light of (3) in proposition \ref{betaificationofhamisotop}, \[\partial_t|\sigma_{t,a}|\leq|\sigma_{0,a}e^{g(x_{0,a})-g(x_{t,a})}| \times | Dg(\partial_tx_{t,a})|= |\sigma_{t,a}|\times |Dg(\partial_t x_{t,a})|\]
with  $|Dg(\partial_t x_{t,a})|\leq |\!|\!|\beta|\!|\!|\times \|D_\xi (f_t\chi_a)\|$, which is bounded from above by some positive constant $C$, and therefore
\[|\sigma_{t,a}|\leq |\sigma_{0,a}|\times e^{tC}.\]
Similar reasoning on $(\Phi^{eq}_t)^{-1}$ yields that there exist a positive constant $C$ such that \[|\sigma_{0,a}|\leq |\sigma_{t,a}|\times e^{tC}.\]
 
Therefore, for any compact $J\subset I$, there is a positive constant $C_J$ such that \[e^{-C_J}|\sigma_{0,a}|\leq |\sigma_{t,a}|\leq e^{C_J}|\sigma_{0,a}|\]

 This implies that there are positive constants s.t. $\dot{SS}(\tilde{K}^{eq})\subset\{|\tau|\leq C|\sigma_{t,a}|\leq C'|\sigma_{0,a}|\leq C''|\sigma_{t,a} \}$ where $\tau$ is the coordinate in the conormal of the time coordinate. By theorem 6.3.1 of \cite{Kashiwara1990SheavesOM}, this is also true of $\dot{SS}(Rj_*\tilde{K}^{eq})$. Therefore, the inclusion $\{t=0\}\subset\R$ is non-characteristic for $Rj_*\tilde{K}^{eq}$. By propositions 3.1.9 and 5.4.13 of \cite{Kashiwara1990SheavesOM}, \[i_0^{-1} Rj_* \tilde{K}^{eq}\simeq R(j\circ i_0)_* \K_{\Delta_{(\R^n\times\R\times\R)^2}\times\{0\}\times(-\infty,1)}.\] This yields $K^{eq}_0\simeq \K_{\Delta_{(\R^n\times\R\times\R)^2}}$. 

Similarly, $\partial_t (\xi_{t,a},\zeta_{t,a})$ is the $\omega$-dual of $-\sigma_{t,a}(\pi_g\circ\rho_{eq})^*D_{x,z} f_{t,a}$ which is equal to $-\sigma_{t,a}D_{x,z} (f_{t,a})_{\pi_g\circ\rho_{eq}}$. Given that $D_{x,z} (f_{t,a})_{\pi_g\circ\rho_{eq}}$ is bounded from above by a constant positive $Cst$, we conclude that \[\|(\xi_{t,a},\zeta_{t,a})-(\xi_{0,a},\zeta_{0,a})\|\leq Cst\times\int_0^t|\sigma_{w,a}|dw,\]

which in turn implies $\|(\xi_{t,a},\zeta_{t,a})-(\xi_{0,a},\zeta_{0,a})\|\leq C_J|\sigma_{0,a}|$ (up to taking a new, bigger, constant $C_J$) from the previous inequalities. This implies that:\begin{align*}
	\dot{SS}(\tilde{K}^{eq})\subset&\{(x_{t,a},\xi_{t,a},z_{t,a},\zeta_{t,a},s_{t,a},\sigma_{t,a}),(x_{0,a},\xi_{0,a},z_{0,a},\zeta_{0,a},s_{0,a},\sigma_{0,a}),(t,\tau)\\&\;:\; C^{\prime-1}_J\|(\xi_{0,a},\zeta_{0,a},\sigma_{0,a})\|\leq\|(\xi_{t,a},\zeta_{t,a},\sigma_{t,a})\|\leq C'_J\|(\xi_{0,a},\zeta_{0,a},\sigma_{0,a})\|\}
\end{align*}
for some constant $C_J'$. Writing $Y=\R^n\times\R\times\R$ and \[E_{C_J'}=\{(y_1,\gamma_1,y_2,\gamma_2,t,\tau):(y_i,\gamma_i)\in T_{y_i}Y_i, C^{\prime-1}_J\|\gamma_1\|\leq \|\gamma_2\|\leq C^{\prime}_J\|\gamma_1\|\},\]
we conclude that
\begin{align*}
\dot{SS}(K^{eq})\subset E_{C_J'}\cap \{|\tau|\leq C|\sigma_1|\leq C'|\sigma_2|\leq C''|\sigma_1|\}\subset T^*Y_1\times T^*Y_2\times T^*I
\end{align*}
 for big enough constants $C$, $C'$ and $C''$. Putting everything together yields:
 \begin{Lem}\label{stupid}\begin{enumerate}
 	\item $K^{eq}_0\simeq \K_{\Delta_{(\R^n\times\R\times\R)^2}}$,
 	\item $\dot{SS}(K^{eq})\cap \{\sigma_2\neq0\}$ is a subset of the graph of $\Phi^{eq}$,
 	\item $\dot{SS}(K^{eq})\subset E_{C_J'}\cap \{|\tau|\leq C|\sigma_1|\leq C'|\sigma_2|\leq C''|\sigma_1|\}$ for three positive constants $C$, $C'$ and $C''$.
 	\item the inclusion $Y\times Y\times\{t\}\subset Y\times Y\times I$ is non-characteristic for $K^{eq}$ for any $t\in I$
 	\end{enumerate}	
 \end{Lem}
 
 We should also note that $(\Phi^{eq}_{t})^{-1}$ can be similarly quantized by a sheaf that we will call $(K^{eq})^\circ$ and also satisfies (mutatis mutandis), the conclusion of lemma \ref{stupid}. Among other things, this implies that $ K^{eq}_0\circ (K^{eq}_0)^\circ\simeq \K_{\Delta_{Y^2}}$.
 
 \paragraph{\bf{Proof of theorem \ref{theor}.}} From the definition of kernel composition, we can use proposition 5.4.14 in \cite{Kashiwara1990SheavesOM} and theorem 1.16 in \cite{guillermou2012microlocaltheorysheavestamarkins} (also found in \cite{Guillermou2014}, result originally due to Tamarkin), as well as lemma \ref{stupid}, to conclude that $SS(K^{eq}\circ(K^{eq})^\circ)\subset \{\tau=0\}$. Therefore, lemma \ref{stupid} and corollary 1.6 in \cite{Guillermou2012SQHIQNP} (an application of propositions 5.4.17 and 2.7.8 of \cite{Kashiwara1990SheavesOM}) yield:
 
 \begin{Lem}\label{stupid2}
 $K^{eq}_t\circ (K^{eq}_t)^\circ\simeq (K^{eq}_t)^\circ\circ K^{eq}_t\simeq \K_{\Delta_{Y^2}}$
 \end{Lem}

 \begin{Rem}\label{duh}
 Let $T_k$ be the translation by $k$ in the $s$ coordinates and write $K'$ the quantization of $T_{-k}\circ\Phi_t^{eq}\circ T_k$.  By construction $K'\simeq (T_k\times T_k\times id)_*K^{eq}$. Moreover, since $T_{-k}\circ\Phi_t^{eq}\circ T_{k}=\Phi_t^{eq}$, the conclusion of lemma \ref{stupid} also holds for $K'$. Therefore, $(K^{eq}_t)^\circ\circ K'_t\simeq \K_{\Delta_{Y^2}}$. This yields: \[K^{eq}_t\simeq K^{eq}_t\circ\big((K^{eq}_t)^\circ\circ K'_t\big)\simeq \big(K^{eq}_t\circ(K^{eq}_t)^\circ\big)\circ K'_t\simeq K'_t.\]
 \end{Rem}
 
To prove the theorem, we are left to prove the following lemma:
 
 \begin{Lem}\label{stupid3}
 Let $A$ be as theorem \ref{theor}. Then, $\dot{SS}(F)\subset A\implies\dot{SS}(K_t^{eq}\circ F)\subset \Phi_t^{eq}(A)$.
 \end{Lem}
 \begin{proof}
It should be noted that $\dot{SS}({K^{eq}})\cap\{\sigma=0\}$ might not be empty. Since, for $i=1,2$, the projection $p_i:Y_1\times Y_2\rightarrow Y_i$ is not proper (on the support), the aforementioned set might contribute to the microsupport of the tensor product $p_{1\#} SS(K^{eq}\overset{L}{\otimes} p_2^{-1}F)$.

By proposition 5.4.14 in \cite{Kashiwara1990SheavesOM} taking $\overline{A}$ the adherence of $A$ in $T^*Y$, we know that
\begin{align*}
	\dot{SS}(K_t^{eq}&\overset{L}{\otimes} p_2^{-1}F)\subset (SS(K_t^{eq})+\overline{A})\cap \dot{T}^*Y \\&\subset\{\big((x_t,\xi_t,z_t,\zeta_t,s_t,\sigma_t),(x_0,\xi'-\xi_0,z_0,\zeta'-\zeta_0,s_0,\sigma'-\sigma_0)\big)\}\\
	&\quad\cup \{\big((x_t,\xi_t,z_t,\zeta_t,s_t,0),(x_0,\xi'-\xi_0,z_0,\zeta'-\zeta_0,s_0,\sigma')\big)\}\\&\quad\cup \{\big((x_t,\xi_t,z_t,\zeta_t,s_t,0),(x_0,-\xi_0,z_0,-\zeta_0,s_0,-\sigma_0)\big)\}=:S
\end{align*} 
 where $(x_0,\xi',z_0,\zeta',s_0,\sigma')$ is in $A$, and if $\sigma_0>0$, \[(x_t,\xi_t,z_t,\zeta_t,s_t,\sigma_t)=\Phi^{eq}_t(x_0,\xi_0,z_0,\zeta_0,s_0,\sigma_0),\] 
and $\big((x_t,\xi_t,z_t,\zeta_t,s_t,0),(x_0,\xi_0,z_0,\zeta_0,s_0,0)\big)\in\dot{SS}(K^{eq}_t)$. From theorem 1.16 in \cite{guillermou2012microlocaltheorysheavestamarkins}, we conclude that $(x,\xi,z,\zeta,s,\sigma)\in q_{1\#}(S)$ if and only if there exists $\big((x^2_n,\xi^2_n,z^2_n,\zeta^2_n,s^2_n,\sigma^2_n),(x^1_n,\xi^1_n,z^1_n,\zeta^1_n,s^1_n,\sigma^1_n)\big)\in S$ such that $\lim_{+\infty}(x^2_n,\xi^2_n,z^2_n,\zeta^2_n,s^2_n,\sigma^2_n)=(x,\xi,z,\zeta,s,\sigma)$ and $\lim_{+\infty}(\xi^1_n,\zeta^1_n,\sigma^1_n)=0$. The sequence $(\xi^1_n,\zeta^1_n,\sigma^1_n)$ is the difference of two sequences $(\xi'_n,\zeta'_n,\sigma'_n)$ and $(\xi_{0,n},\zeta_{0,n},\sigma_{0,n})$ such that
$(x^1_n,\xi'_n,z^1_n,\zeta'_n,s^1_n,\sigma'_n)\in A$ and $(x^1_n,\xi_{0,n},z^1_n,$ $\zeta_{0,n},s^1_n,\sigma_{0,n})\in \dot{T}^*Y$ such that $\big((x^2_n,\xi^2_n,z^2_n,\zeta^2_n,s^2_n,\sigma^2_n),$ $(x^1_n,\xi_{0,n},z^1_n,\zeta_{0,n},$ $s^1_n,\sigma_{0,n})\big)$ $\in\dot{SS}(K^{eq})$. Write $u_n=\|(\xi'_n,\zeta'_n,\sigma'_n)\|$. From the hypothesis on $A$, we know that $A$ is the pullback of a closed conic set $A'$ in $\dot{T}^*(\R^n\times\Sph^1\times\Sph^1)$ via a projection $proj$ such that the  projection of $A'$ on $\R^n$ is compact. This implies that $Dproj(x^1_n,\xi'_n/u_n,z^1_n,\zeta'_n/u_n,s^1_n,\sigma'_n/u_n)$ converges to $(x_\infty,\xi_\infty,\theta_\infty,\zeta_\infty,\theta'_\infty,\sigma_\infty)$ up to taking a subsequence. Finally, note that $A'\subset\{\sigma>0\}$.

Let us consider $(x,\xi,z,\zeta,s,\sigma)\in q_{1\#}(S)\cap \dot{T}^*Y$. Assume that $\sigma=0$, we will show that this assumption cannot hold. When $n\rightarrow +\infty$, $\sigma^2_n\rightarrow0$ while $(\xi^2_n,\zeta^2_n)$ doesn't converge to $0$. From the estimates in lemma \ref{stupid}, this implies that $\sigma_{0,n}\rightarrow 0$ and $(\xi_{0,n},\zeta_{0,n})$ doesn't converge to $0$. Therefore, $\sigma'_n$ converges to $0$ while $(\xi'_{n},\zeta'_{n})$ doesn't. This implies that $\sigma_\infty=0$ while $(\xi_\infty,\zeta_\infty)\neq 0$. Absurd since $A'$ is closed and is a subset of $\{\sigma>0\}$.

If $\sigma<0$, then $\sigma_{0,n}<-\epsilon$ for some $\epsilon>0$ and any $n$ big enough, and therefore $\sigma_{0,n}-\sigma_{n}'$ cannot converge to $0$ (absurd).

Therefore $\sigma>0$, which implies that $(x,\xi,z,\zeta,s,\sigma)\in\Phi_t(A)$.
 \end{proof}
 
  \begin{Rem}
 	The arguments of the proof can be straightforwardly adapted to show that $\dot{SS}(F)\subset \Phi^{eq}_t(A)\implies \dot{SS}((K_t^{eq})^\circ\,\circ F)\subset (\Phi^{eq}_t)^{-1}\circ\Phi^{eq}_t(A)=A$
 	\end{Rem}
 
 \begin{proof}[proof of theorem \ref{theor}]
 	From lemma \ref{stupid3}, we know that $K_t^{eq}\circ\underline{\quad}$ is a morphism from $D_A(Y)$ to $D_{\Phi^{eq}_t(A)}(Y)$, and from lemma \ref{stupid2}, we know that this morphism has an inverse given by $(K_t^{eq})^\circ\circ\underline{\quad}$.
 	
 	Moreover, from lemma \ref{stupid}, given $i_t:Y\times\{t\}\rightarrow Y\times I$ the inclusion and \[A'= \{(\Phi_t^{eq}(x, \xi,z,\zeta), (t; -f_t^{eq}(\Phi_t^{eq}(x, \xi,z,\zeta)))); (x, \xi,z,\zeta)\in A, s \in I\},\]
 	we can easily verify that $i_t$ is non-characteristic for $A'$. From the proof of corollary 2.1.5 in \cite{Guillermou2014}, we conclude that \[\text{Hom}(F,G)\overset{\sim}{\leftarrow}\text{Hom}(K^{eq}\circ F,K^{eq}\circ G)\overset{\sim}{\rightarrow}\text{Hom}(K_t^{eq}\circ F,K_t^{eq}\circ G).\]
 \end{proof}

\begin{Rem}\label{duh2}
	If a sheaf $F$ has microsupport that lies in $\{\sigma>0\}$, that is periodic  in the $z$ coordinate, and is compact in the $x$ coordinates, one may consider $\bigoplus_{k\in\Z}T_{k*}F$ where $T_k$ is the translation by $k$ in the $s$ coordinates. If $\dot{SS}(\bigoplus_{k\in\Z}T_{k*}F)\subset\{\sigma>0\}$, then $\dot{SS}(\bigoplus_{k\in\Z}T_{k*}F)$ respects the conditions of  theorem \ref{theor} and $e(\bigoplus_{k\in\Z}T_{k*}F)=e(F)$. Moreover, by remark \ref{duh}, $K_t^{eq}\circ\left(\bigoplus_{k\in\Z}T_{k*}F\right)\simeq \bigoplus_{k\in\Z}T_{k*}\left(K_t^{eq}\circ F\right)$.
\end{Rem}

	\bibliographystyle{plain}
	\bibliography{./biblio.bib}
\end{document}